\documentclass{siamart190516}
\usepackage{amsmath,amssymb,amsfonts,epsfig,color,url}

\usepackage{hyperref} 

\usepackage{graphicx}
\usepackage{float}
\usepackage[caption=false]{subfig}

\headers{Impact of advection on large-wavelength stability}{J. Yang, J. D. M. Rademacher, E. Siero}

\graphicspath{{figures/}}

\newsiamremark{remark}{Remark}
\newsiamremark{hypothesis}{Hypothesis}
\crefname{hypothesis}{Hypothesis}{Hypotheses}
\newsiamthm{claim}{Claim}

\definecolor{colorJRblue}{rgb}{0.,0.,1.}
\definecolor{colorJRred}{rgb}{1.,0.,0.}

\def\R{\mathbb{R}}
\def\C{\mathbb{C}}
\def\Z{\mathbb{Z}}

\def\Lspace{{\sf L}}
\def\Hspace{{\sf H}}

\def\crit{\mathrm{crit}}
\def\per{{\rm per}}
\DeclareMathOperator{\range}{\mathrm{range}}

\def\E{E}
\def\rms{{\rm s}}
\def\tlam{{\tilde\lambda}}

\def\kap{{\tilde\kappa}}

\def\nl{{\rm nl}}
\def\crit{{\rm c}}

\def\st{{\rm st}}
\def\hom{{\rm hom}}

\def\hK{k_0}
\def\hq{q_2}
\def\tq{q_0}

\def\rmi{\mathrm{i}}
\def\rmd{\mathrm{d}}
\def\rme{\mathrm{e}}

\def\calO{\mathcal{O}}
\def\calL{\mathcal{L}}

\def\calT{\mathcal{T}}

\def\calR{\mathcal{R}}
\def\calZ{\mathcal{Z}}
\def\calE{\mathcal{E}}
\def\calB{\mathcal{B}}

\def\Id{{\rm Id}}
\def\A{L}

\def\M{M}
\def\B{B}
\def\Q{Q}
\def\K{K}

\def\x{{\bf x}}
\def\kc{{{\bf k}_{\rm c}}}
\def\kcsq{{{\bf k}_{\rm c}^2}}
\def\ta{{b}}

\def\tQ{\Q_0}
\def\hQ{\Q_2}

\def\bbeta{{\beta\beta}}

\def \talpha {{\tilde\alpha}}
\def \calpha {{\check\alpha}}

\def\Re{\mathrm{Re}}

\def \sgn {\mathrm{sgn}}

\def \dif {\mathrm{d}}
\def \diag {\mathrm{diag}}

\def\zz{{\rm zz}}
\def\eh{{\rm eh}}

\begin{document}

\title{The impact of advection on large-wavelength stability of stripes near planar Turing instabilities\thanks{December 24, 2019. \funding{This work was funded by the China Scholarship Council, and Degree completion stipend from University of Bremen.}}}

\author{Jichen~Yang\thanks{University of Bremen, Faculty 3 -- Mathematics, Bibliothekstrasse 5, 28359  Bremen, Germany (\email{jyang@uni-bremen.de}).} \and Jens~D.~M.~Rademacher\thanks{University of Bremen, Faculty 3 -- Mathematics, Bibliothekstrasse 5, 28359  Bremen, Germany (\email{jdmr@uni-bremen.de}).} \and Eric~Siero\thanks{Carl von Ossietzky University of Oldenburg, Institute for Mathematics, Carl von Ossietzky Str. 9-11, 26111 Oldenburg, Germany (\email{eric.siero@uni-oldenburg.de}).}}

\maketitle

\begin{abstract}
It is well known that for reaction-diffusion systems with differential isotropic diffusions, a Turing instability yields striped solutions. In this paper we study the impact of weak anisotropy by directional advection on such solutions, and the role of quadratic terms. We focus on the generic form of planar reaction-diffusion systems with two components near such a bifurcation. Using Lyapunov-Schmidt reduction and Floquet-Bloch decomposition we derive a rigorous parameter expansion for existence and stability against large wavelength perturbations. This provides detailed formulae for the loci of bifurcations and so-called Eckhaus as well as zigzag stability boundaries under the influence of the advection and quadratic terms. In particular, while destabilisation of the background state is through modes perpendicular to the advection (Squire-theorem), we show that stripes can bifurcate zigzag unstably. We illustrate these results numerically by an example. Finally, we show numerical computations of these stability boundaries in the extended Klausmeier model for vegetation patterns and show stripes bifurcate stably in the presence of advection.
\end{abstract}

\begin{keywords}
  Reaction-diffusion systems, striped solutions, Turing instability, advection, large-wavelength stability, Eckhaus, zigzag
\end{keywords}

\begin{AMS}
  35B10, 35B35, 35B36, 35K57
\end{AMS}

\section{Introduction}\label{s:intro}

It is well known that from the ubiquitous spatially isotropic Turing instabilities various patterned solutions bifurcate. In one dimension the basic spatially periodic ones are wavetrains, which trivially extend to stripe solutions in two space dimensions, where they are in competition with hexagonal and square shaped states, e.g., \cite{Hoyle2006}. The question arises, which pattern is selected at onset of the instability. It is well known that in the isotropic situation stripes are unstable with respect to modes on the hexagonal lattice near onset in the presence of a generic quadratic term in the nonlinearity. In \cite{Gowda2014} this has been discussed in the context of vegetation patterns. In contrast, it has been found in \cite{Siero2015} that in a sloped terrain, the vegetation patterns, i.e. the stripes, are stable at onset. Here the slope is modelled by an advective term in the water component, which breaks the spatial isotropy. Indeed, from a symmetry perspective for weakly anisotropic perturbations and on the hexagonal lattice this has been predicted already in \cite{Callahan2000}. The destabilising effect of advection terms on homogeneous states have been broadly studied in the context of differential flows, e.g., \cite{Rovinsky1992,Merkin2000,Carballido-Landeira2012} and also appear in ecology, e.g., \cite{Wang2009,Cangelosi2015,Bennett2019}, where we believe our results can also be useful.

In this paper we study the stability of stripes in reaction-diffusion systems for weak anisotropy regarding large wavelength perturbations; a forthcoming paper will consider hexagonal modes. We are particularly interested in refining the results of \cite{Siero2015} which indicate a stabilising effect of advection for stripes aligned with this. In particular, it was proven that the onset of instability of the homogeneous state, i.e., the nature of the Turing-Hopf instability, is due to one-dimensional modes (a `Squire'-theorem). However we shall explain below that this does not necessarily imply stability of bifurcating stripes.

Large wavelength modes, also called sideband modes, are well understood in one space dimension through the Ginzburg-Landau formalism, e.g., \cite{Hoyle2006,Doelman2009,Schneider2017}, most directly from the fourth order Swift-Hohenberg equation. Here only sideband modes are relevant and the so-called Eckhaus region describes the stability boundary, which is crossed when stretching or compressing the wavetrains too much. In two space dimensions, instabilities along the stripe that is formed by trivially extending the one-dimensional wavetrain, become additionally relevant. The large wavelength modes of this type give rise to the so-called zigzag stability boundary. It is well known that for the Swift-Hohenberg this is crossed when wavetrains are stretched by any amount in the isotropic case, but  detailed rigorous studies for reaction diffusion systems (even without advection) seem scarce; in \cite{Pena2003} a reduction to nearly hexagonal lattices is applied. Indeed, zigzag stability can also be studied with the aid of a modulation equation, the so-called Newell-Whitehead-Segal equation, again most directly linked with the Swift-Hohenberg equation \cite{Hoyle2006}. 

\medskip
In this paper we take a direct approach and first study the existence of stripes with detailed expansions by Lyapunov-Schmidt reduction and then analyse the large wavelength stability via Floquet-Bloch decomposition in the spirit of \cite{Mielke1995,Rademacher2007,Doelman2009}. The advantages of this approach are that it is fully rigorous and that we gain direct access to all relevant characteristic quantities in terms of the advection, the quadratic terms, stretching and compressing. A particular motivation is to augment the discussion of stripe stability in \cite{Siero2015} for a variant of the Klausmeier model, where small advection and zigzag modes were not considered in any detail, cf.\ \S\ref{s:Klausmeier}.

The approach applies to arbitrary number of components, but the parameter spaces and determination of signs of relevant characteristics become analytically less accessible for more than two components. Hence we restrict our attention to this case.

\medskip
Upon changing coordinates, the generic form of such a system up to cubic nonlinearity reads
\begin{align}\label{e:RDS}
u_t = D\Delta u + \A u + \calpha \M u + \beta \B u_x + \Q[u,u] + \K[u,u,u], \; \x\in\R^2
\end{align}
with multilinear functions $\Q, \K$ and diagonal diffusion matrix $D>0$; higher order nonlinear terms can be added without change to our results near bifurcation. We assume that for $\calpha=\beta=0$ the zero steady state is at a Turing instability with wavenumber $\kc$, cf.\ Definition \ref{def:Turing} below, and that $\calpha$ moves the spectrum through the origin. The isotropy is broken for $\beta\neq 0$, and we assume differential advection 
\[
\B=\B(c)=\begin{pmatrix}1+c & 0\\ 0 & c\end{pmatrix}, \; c\in\R,
\]
which can be realised under the natural assumption of uni-directional anisotropy. Note that $\beta c\partial_x$ appears in both equations as a comoving frame in the $x$-direction.

\medskip
Our main results may be summarised as follows. Here the parameters are $\mu=(\alpha,\beta,\kap)$ where $\alpha=\lambda_M\calpha$ for certain $\lambda_M\neq0$ and $\kap=\kappa-\kc$ is the deviation of the stripes' nonlinear wavenumber from $\kc$, i.e., the stripes' spatial period is $2\pi/\kappa$. Throughout we consider $|\mu|\ll 1$, and consider stripes  $U_\rms(x;\mu)$ that are constant in $y$ with amplitude parameter $A=\|\widehat{U_\rms}(1;\mu)\|$ the norm of the first Fourier mode.

\medskip
\paragraph{Existence of stripes (Theorem~\ref{t:bif})} 
The existence of striped solutions $U_\rms(x;\mu)$ to \eqref{e:RDS} with small amplitude $A$ near the onset of the Turing instability is equivalent to solving an algebraic equation
\[
\alpha + \rho_\beta \beta^2 + \rho_\kap \kap^2 + \rho_\nl A^2=0,
\]
where $\rho_\beta,\rho_\kap$ are determined by the linearisation in $u=0$, and $\rho_\nl$ involves the nonlinear terms. We have $\rho_\beta>0$, $\rho_\kap<0$ so that the bifurcation loci form a hyperbolic paraboloid, and in the supercritical case $\rho_\nl<0$ the corresponding amplitudes $A$ follow a family of supercritical pitchfork bifurcations, cf.\ Fig.~\ref{f:sidebandintro} (dashed curves). We provide an expansion of $U_\rms(x;\mu)$ in the parameters and the velocity parameter $c$ is a function of $\mu$ that is to leading order affine in $\alpha, \kap$. For direction of the stripe motion $\beta c$ we have $\sgn(c) = - \sgn(a_1)$, i.e., in case the first component is an inhibitor the motion is with $\beta$, and it is opposite $\beta$ if it is an activator.

\begin{figure}[t]
\centering
\subfloat[]{\includegraphics[width=0.45\linewidth]{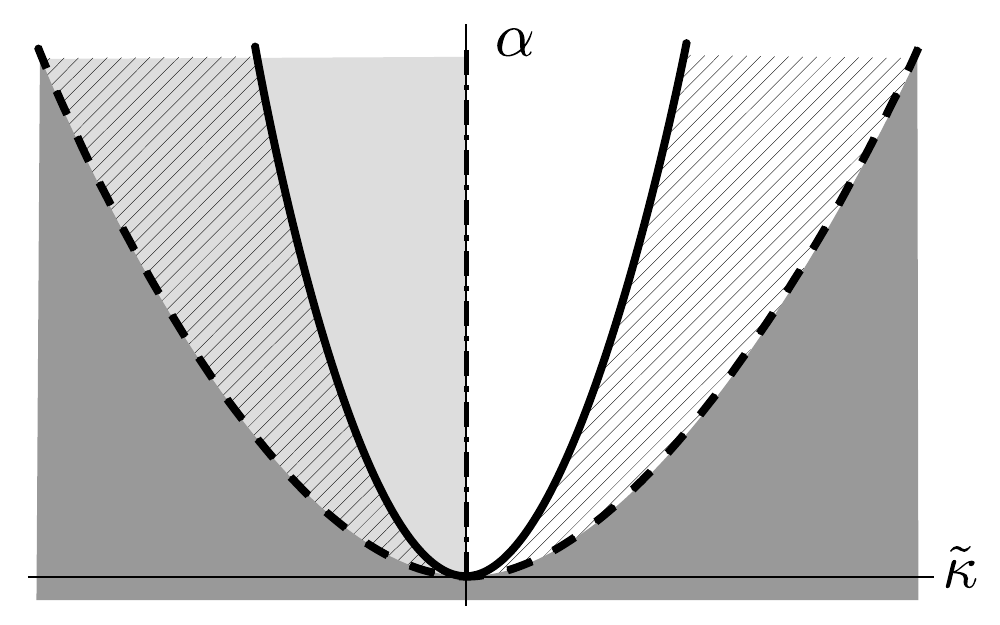}\label{f:sidebandintrobet0}}
\hfil
\subfloat[]{\includegraphics[width=0.45\linewidth]{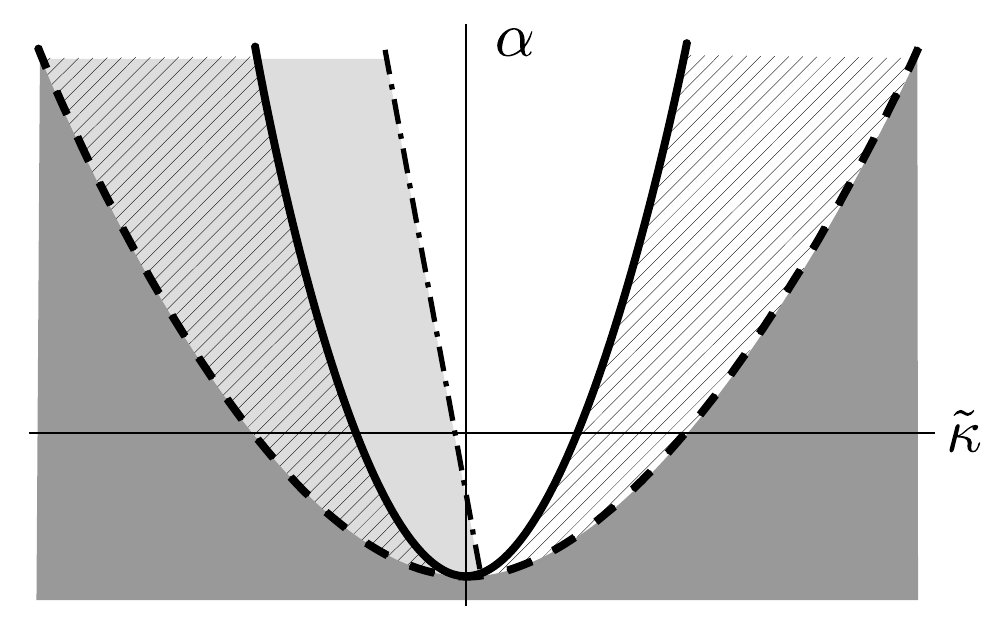}\label{f:sidebandintrobetn0}}
\caption{We plot sketches of the leading order existence and stability boundaries near the Turing bifurcation point at the origin in (a). Stripes exist in the complement of the dark grey regions. Hatched region: Eckhaus unstable. Light grey region: zigzag unstable. (a) $\beta=0$, $\M=\Id$, $\Q=0$, (b) sample for $\beta\neq0$, $\M\neq \Id$, $\Q\neq 0$. Note in (b) the existence and Eckhaus boundaries are shifted downwards, the zigzag boundary is tilted and the attachment point to the existence shifted.
}
\label{f:sidebandintro}
\end{figure}

\medskip
Having established the existence of stripes, we obtain the following results concerning large-wavelength stability. The stability or instability of stripes against the large wavelength perturbations parallel and orthogonal to the stripes, is referred to as {\em zigzag} and {\em Eckhaus} in/stability, respectively. Our results on these two types of instability at the onset of Turing bifurcation may be summarised as follows, cf. Theorem \ref{t:zigzag}\,\&\,\ref{t:Eckhaus}. 

\medskip
\paragraph{Zigzag instability}

We determine the leading order curvature of the spectrum for modes parallel to the stripes as
\[
\kc \rho_\kap\kap + \rho_{\calpha} \alpha + \rho_\bbeta \beta^2,
\]
which means zigzag instability for a positive value. For $\rho_\calpha=0$ the leading order zigzag boundary is independent of $\alpha$ as in the aforementioned isotropic case of Swift-Hohenberg. It turns out that $\rho_\calpha=a+b$, where $a=0$ if $M=\Id$ and $b=0$ if $Q=0$, which highlights the impact of non-trivial $M$ and the quadratic term. The sign of $\rho_\bbeta$ determines whether $\beta$ has a stabilising or destabilising effect, and we determine this for `small' $Q$. It turns out that if the first component is the inhibitor, $a_1<0$, then $\rho_\bbeta<0$. However, the different combinations of signs allow to move and tilt the zigzag boundary, cf. Fig.~\ref{f:zzfull}. In particular, it is possible that stripes are zigzag unstable at onset, which shows a limitation of the `Squire theorem' of \cite{Siero2015}, i.e., the fact that a {\em homogeneous} steady state is always destabilised by modes that are constant in the direction perpendicular to the advection. 

\medskip
\paragraph{Eckhaus instability}
To leading order the curvature of the spectrum for modes orthogonal to the stripe has the sign of
\[
-(\talpha + 3\rho_\kap \kap^2),
\] 
where $\talpha=\alpha+\rho_\beta \beta^2$ is the deviation from the bifurcation loci. Hence, in terms of $\talpha$ the leading order curvature is independent of the advection $\beta$, and just according to the well-known Eckhaus boundary as a function of $\kap$, cf.\ Fig.~\ref{f:sidebandintro} (solid curves). Thus, in contrast to the zigzag instability, relative to the bifurcation loci there is no leading order impact of the advection on this large wavelength stability.
Nevertheless, for fixed unfolding parameter $\alpha$ the interval of stable $\kap$ is larger, i.e., stripes are more resilient to stretching/compressing compared to the isotropic case.

\bigskip
This paper provides a first step to understand analytically and in detail the stability of stripes in \eqref{e:RDS} under the influence of advection on the plane $\x\in\R^2$. As mentioned, the natural next step is to study stability on lattices, in particular (near) hexagonal lattices as, e.g., in \cite{Callahan2000,Gowda2014}. This is the subject of a forthcoming paper, where we give detailed expansions of the different stability boundaries that arise under influence of the advection and quadratic terms.

\medskip
This paper is organised as follows: In \S\ref{s:Turing} we discuss linear stability of the homogeneous state near the Turing instability as a preparation for the analysis of stripes. The existence of stripes is studied in \S\ref{s:bif}, and in \S\ref{s:stability} we study the large wavelength in/stabilities, i.e., zigzag and Eckhaus in/stabilities. In \S\ref{s:example}, we illustrate these results by a concrete example of the form \eqref{e:RDS} and in \S\ref{s:Klausmeier}, we study the large wavelength instabilities numerically for the extended Klausmeier model that was used in \cite{Siero2015}.

\section{Turing instability}\label{s:Turing}

The linearisation of \eqref{e:RDS} in $u_{\rm hom}=0$ is 
\[
\calL=D\Delta+\A + \calpha \M + \beta \B\partial_x,
\]
whose spectrum is most easily studied via the Fourier transform
\[
\hat{\calL}(k,\ell) = -(k^2+\ell^2)D+\A + \calpha \M + \rmi k\beta \B,
\]
with Fourier-wavenumbers $k$ in $x$-direction and $\ell$ in $y$-direction. It is well known, e.g., \cite{Sandstede2002},  that in the common function spaces such as $\Lspace^2(\R^2)$ the spectrum $\Sigma(\calL)$ of $\calL$ equals that of $\hat\calL$ and is the set of roots of the (linear) dispersion relation
\begin{align}\label{e:disp}
d(\lambda,k,\ell) = \det(\hat\calL(k,\ell) -\lambda\Id).
\end{align}
Let $S_{\kc}\subset\R^2$ be the circle of radius $\kc$.

\begin{definition}\label{def:Turing}
We say that $\calpha=\beta=0$ is a (non-degenerate) Turing instability point for $u_{\rm hom}$ in \eqref{e:RDS} with wavelength $\kc$ if 
\begin{itemize}
\item[(1)]  $\A$ has strictly stable spectrum $\Sigma(\A)\subset\{\lambda\in\C : \Re(\lambda)<0\}$,
\item[(2)] The spectrum of $\calL$ is critical for wavevectors $(k,\ell)$ of length $\kc>0$: 
\[
d(\lambda,k,\ell) =0 \ \&\ \Re(\lambda)\geq 0\quad \Leftrightarrow\quad \lambda=0,\ (k,\ell)\in S_{\kc}
\]
which in particular means $\Sigma(\calL)\cap \{z\in \C: \Re(z)\geq 0\} = \{0\}$,
\item[(3)] $\partial_\lambda d\neq 0$ at $\lambda=0$ and $(k_\crit,\ell_\crit)\in S_{\kc}$. We denote the unique continuation of these solutions to \eqref{e:disp} by $\lambda_\crit(k,\ell,\mu)$, i.e., $(k,\ell)$ in a neighboorhood of $S_{\kc}$.
\end{itemize}
\end{definition}
Writing $\A=\begin{pmatrix}a_1 & a_2\\ a_3 & a_4\end{pmatrix}$, condition (1) implies negative trace of $\A$, $a_1+a_4<0$, and positive determinant $a_1a_4>a_2a_3$, and (3) implies the well known condition $d_1a_4+d_2a_1>0$, which together imply $a_2a_3<a_1a_4<0$, e.g., \cite{Murray2003}.

As a first step to understand the impact of advection,  the next lemma shows that, for this two-component case, the unfolding by $\beta$ is only to quadratic order.

\begin{lemma}\label{l:Turbeta} 
For the critical eigenvalues near a Turing instability of \eqref{e:RDS} as in Definition \ref{def:Turing} it holds for any  $(k_\crit,\ell_\crit)\in S_{\kc}$ that
\[
\lambda_\crit(k_\crit,\ell_\crit;\beta)=\rmi k_\crit (\lambda_{\beta}+c) \beta + k_\crit^2\lambda_{\bbeta} \beta^2 + \calO(|k_\crit\beta|^3),
\]
where $\lambda_{\beta}=\frac{a_4-\kcsq d_2}{a_1+a_4-\kcsq(d_1+d_2)}$, $\lambda_{\bbeta}= \frac{(a_1-\kcsq d_1)(a_4-\kcsq d_2)}{(a_1+a_4-\kcsq(d_1+d_2))^3}>0$. In particular, the real part grows fastest for 1D-modes with $\ell_\crit=0$ and remains zero for transverse modes with $k_\crit=0$.
\end{lemma}

\begin{proof}
This follows immediately from the next lemma upon setting $\delta=k_\crit \beta$, $\ta_1 = -(k_\crit^2+\ell_\crit^2)d_1 + a_1$, $\ta_2=a_2$, $\ta_3=a_3$, $\ta_4 = -(k_\crit^2+\ell_\crit^2)d_2 + a_4$ and shifting by $\rmi k_\crit \beta c$. The last statement of the lemma is simply a consequence of the fact that the largest value real part of $\lambda_\crit$ is attained at the largest value of $k_\crit^2$, which occurs at $\ell_\crit=0$ since $k_\crit^2+\ell_\crit^2=\kcsq $.
\end{proof}

\begin{remark}\label{r:Turing}
The lemma in fact proves  the plots in Figure 2 of \cite{Siero2015} near onset. 
It is well known that for a two-component system $\kcsq = \frac{d_1 a_4+d_2 a_1}{2d_1d_2}$ and $a_2a_3= (a_1-\kcsq d_1)(a_4-\kcsq d_2)$.
\end{remark}

\begin{lemma}\label{l:lambetamata} 
For a matrix $\begin{pmatrix}\ta_1+\rmi \delta & \ta_2\\ \ta_3 & \ta_4\end{pmatrix}$ with $b_1\neq 0$ and simple zero eigenvalue, the expansion of that eigenvalue reads
\[
\lambda(\delta)=\rmi \lambda_| \delta + \lambda_{||} \delta^2 + \calO(|\delta|^3),
\]
where $\lambda_|=\frac{\ta_4}{\ta_1+\ta_4}$, $\lambda_{||}= \frac{\ta_1\ta_4}{(\ta_1+\ta_4)^3}$. 
\end{lemma}

\begin{proof} Straightforward implicit differentiation, expansion of characteristic polynomial and use of assumptions, which in particular imply $(b_1+b_4)b_1\neq 0$.
\end{proof}

Note that $b_2=0$ or $b_3=0$ is not possible due to the assumption $b_1\neq0$ and $b_1b_4=b_2b_3$.

\begin{remark}\label{r:ev}
For the matrix $\begin{pmatrix}\ta_1 & \ta_2\\ \ta_3 & \ta_4\end{pmatrix}$ in Lemma~\ref{l:lambetamata}, i.e., $\delta=0$, we can choose the kernel eigenvector $\E_0$ and the adjoint kernel eigenvector $\E_0^*$ with $\langle \E_0,\E_0\rangle=1$ and $\langle \E_0,\E_0^* \rangle=1$ as
\begin{align*}
E_0 = (b_2,-b_1)^T/c_0, \quad E_0^* = (b_3,-b_1)^T/c_0^*,
\end{align*}
with $c_0 := \sqrt{b_2^2 + b_1^2}$, $c_0^* := (b_2b_3 + b_1^2)/c_0$. Here $c_0^*\neq0$ since $b_1^2+b_2b_3 = b_1^2+b_1b_4 = b_1(b_1+b_4)\neq0$.
\end{remark}

\medskip
In contrast to $\beta$, the change of real parts of the critical eigenvalue through $\calpha$, with matrix $M=(m_{ij})_{1\leq i,j\leq 2}$, is linear with coefficient 
\begin{equation}
\begin{aligned}
\lambda_\M&:=-\left.\frac{\partial_\calpha d}{\partial_\lambda d}\right|_{\calpha=0,\lambda=0}\\
&=\frac{m_{11}(a_4-\kcsq d_2)-m_{12}a_3-m_{21}a_2+m_{22}(a_1-\kcsq d_1)}{a_1+a_4-\kcsq (d_1+d_2)}\neq 0,
\end{aligned}
\end{equation}
where we \emph{assume} $\lambda_M\neq 0$ throughout this paper. Notably, $\lambda_\M=1$ if $\M=\Id$ in which case $\calpha$ just rigidly moves the real part of the spectrum.

In the following we therefore change parameters and use the effective impact on the real part given by 
\[
\alpha:=\lambda_\M\calpha
\]
as the new parameter so that
\begin{equation}\label{e:lamcrit1}
\begin{aligned}
\lambda_\crit(k_\crit,\ell_\crit;\alpha,\beta) = \ &\alpha + \rmi (k_\crit (\lambda_{\beta}+c) + a_\M\lambda_{\M\beta}\alpha)\beta + k_\crit^2\lambda_{\bbeta} \beta^2\\
& + \calO(a_\M\alpha^2 + |k_\crit\beta|^3),
\end{aligned}\end{equation}
with $\lambda_{\M\beta} :=  k_\crit\frac{m_{22}-\lambda_\M-(2\lambda_\M-m_{11}-m_{22})\lambda_{\beta}}{\lambda_\M(a_1+a_4-\kcsq(d_1+d_2))}$, and we emphasise the special case $M=\Id$ through the factor $a_M$, where $a_M=0$ if $M=\rm{Id}$ and $a_M=1$ otherwise. 

\medskip
Here we highlight an a priori consequence for the $\Lspace^2(\R^2)$-stability of striped solutions $U_\rms$ with wavenumber $\kappa=\kc+\kap$ that are oriented orthogonal to the $x$-direction, i.e., $\partial_y U_\rms\equiv 0$. We assume (and prove in the next section) the existence of a curve of such striped solutions $U_\rms(x;\tau)$ parametrised by $\tau\in[0,\tau_0)$ for some $\tau_0>0$, with $U_\rms(x;0)=0$, and corresponding parameter curve $\mu(\tau)=(\alpha,\beta,\kap)(\tau)$ with $\beta(0)\neq0$, $|\mu(0)|\ll1$, and velocity parameter $c(\tau)$. 

\begin{corollary}\label{c:stable}
For $0<\tau\ll1$ the spectral stability in $\Lspace^2(\R^2)$ of $U_\rms$ is entirely determined by large-wavelength modes, i.e., if $U_\rms$ is zigzag and Eckhaus stable then it is spectrally stable in $\Lspace^2(\R^2)$.
\end{corollary}

In particular, a family with constant $\kap=0$, i.e., stripes with wavenumber $\kc$, bifurcates stably, if it is zigzag-stable at onset.

\begin{proof}
Since $\beta(0)\neq 0$, by Lemma~\ref{l:Turbeta}, see also the Squire-theorem~\cite[Theorem 2]{Siero2015}, the spectrum of $U_\rms(x;0)=u_\hom$ with parameters $\mu(0), c(0)$ is critical only for $\kap,\ell\approx 0$. More precisely, for all sufficiently small $\epsilon>0$ there is $\delta>0$ such that $\Re(\lambda_\crit(\kc+\kap,\ell;\alpha(0),\beta(0)))<-\delta$ for all $\kap,\ell$ with $|\kap|, |\ell|>\varepsilon$. It suffices to show that the same holds for the spectrum of the linearisation $\calL_\st$ of \eqref{e:RDS} in $U_\rms$ for any sufficienly small $\tau$.

Via Floquet-Bloch decomposition, the spectrum of $\calL_\st$ can be encoded in a complex analytic dispersion relation $d_\st(\lambda,\gamma,\ell)$, $\gamma\in[0,2\pi)$, e.g., \cite{Mielke1995,Rademacher2007,Doelman2009}, and \S\ref{s:stability}. Since $\calL_\st(0)=\calL$ roots of $d_\st$ converge locally uniformly in $\C$ to roots of $d$ for $k=2\pi m + \gamma$ with suitable $m\in\Z$. Hence, any spectrum that is bounded away from $\rmi\R$ for $u_\hom$ will be bounded away from $\rmi\R$ for all sufficiently small $\tau$.
\end{proof}

\section{Bifurcation of stripes}\label{s:bif}
Stripes are travelling waves solutions of \eqref{e:RDS} that are constant in $y$ and for any $t$ periodic in $x$. In order to determine the bifurcation of stripes it thus suffices to consider the 1D case $\x=x\in[0,2\pi/\kappa]$ with periodic boundary conditions and wavenumber $\kappa$. The definition of a Turing instability point implies that $\calL$ restricted to 1D possesses a kernel at $\alpha=\beta=0$ on spaces of $2\pi/\kc$-periodic functions and upon unfolding also for nearby periods. Let us therefore rescale space and consider periodic boundary conditions on $[0,2\pi]$. This modifies the linear part \eqref{e:RDS} to 
\[
\calL_\mu:= \kappa^2 D\partial_x^2  + \A + \calpha \M  + \beta \kappa\B \partial_x
\]
with the off-critical parameter $\kap$ in $\kappa= \kc+\kap$ that allows to detects stripes with nearby wavenumber. We recall the parameter vector $\mu=(\alpha,\beta,\kap)$. By Lemma \ref{l:Turbeta}, \eqref{e:lamcrit1}, and straightforward generalisation to include $\kap$, the continuation of the zero eigenvalue of $\calL_\mu$ has an expansion
\begin{equation}\label{e:evlinearop}
\begin{aligned}
	\lambda_\mu =&\; \alpha + \rho_\beta\beta^2+\rho_\kap\kap^2 + \rmi(\gamma_\beta+\gamma_{\kap\beta}\kap + a_\M\lambda_{\M\beta}\alpha)\beta \\ 
	& +a_\M \lambda_{\M\kap}\alpha\kap+ \calO(a_\M\alpha^2 + |\kap|^3+ |\beta|^3),
\end{aligned}
\end{equation}
where again $a_M=0$ if $M=\rm{Id}$ and $a_M=1$ otherwise. The coefficients are
\[
\rho_\beta=\kcsq\lambda_{\bbeta}>0, \quad \gamma_\beta=\kc (\lambda_{\beta}+c), 
\]
as in Lemma~\ref{l:Turbeta} and with $\gamma_{\kap\beta}=\lambda_{\kap\beta}+c$, the dispersion relation $d(\lambda,k;\mu)=0$ as well as $\partial_k \lambda_\crit(\kc;0)=0$ yields
\begin{align*}
\lambda_{\kap\beta} &=  \rmi\left.\frac{\partial_{k\lambda}d\cdot\partial_\beta\lambda+\partial_k^2d}{\partial_\lambda d}\right|_{k=\kc,\mu=0,\lambda=0} \in \R,\\
\lambda_{\M\kap} &=  -\left.\frac{\lambda_\M\partial_{k\lambda}d+\partial_k^2d}{\lambda_\M\partial_\lambda d}\right|_{k=\kc,\mu=0,\lambda=0} \in \R,\\
\rho_\kap &= -\left.\frac{\partial_k^2 d}{2\partial_\lambda d}\right|_{k=\kc,\lambda=0}<0
\end{align*}
with the last sign due to $d_1a_4+d_2a_1>0$, $a_1+a_4<0$ and
\[
\rho_\kap = -\frac{d_1a_4+d_2a_1-6d_1d_2\kcsq}{a_1+a_4-(d_1+d_2)\kcsq}
= \frac{2(d_1a_4+d_2a_1)}{a_1+a_4-(d_1+d_2)\kcsq}.
\]

Vanishing real part $\Re(\lambda_\mu) = 0$ thus occurs to leading order on a hyperbolic paraboloid 
\[
\alpha =\calB(\kap,\beta)= -(\rho_\kap\kap^2 +\rho_\beta\beta^2)
\] 
in $\mu$-space. Since the eigenvalue is stable (unstable) for $\alpha<\calB(\kap,\beta)$ ($\alpha>\calB(\kap,\beta)$), this constitutes the bifurcation surface at leading order. 

The next theorem specifies the bifurcation and expansion of the stripe solutions near $\mu=0$, where our main point is the effect of $\beta$ and its interaction with $\alpha, \kap$. Rather than expanding with abstract coefficients, we provide explicit formulae evaluated at $\mu=0$ in terms of the following quantities.
\begin{equation}\label{e:defs} 
\begin{aligned}
\tQ &:= -2\A^{-1} \Q[\overline{\E_0},\E_0],\quad
\hQ := -2(-4 \kcsq D + \A)^{-1} \Q[\E_0,\E_0],\\
\tq&:= \langle \Q[\E_0, \tQ],\E_0^* \rangle,\quad 
\hq:= \langle \Q[\overline{\E_0}, \hQ],\E_0^* \rangle,\\
\hK &:= \langle \K[\E_0,\E_0,\overline{\E_0}], \E_0^*\rangle, \quad 
\rho_\nl:= 3\hK + 2\tq + \hq,\\
w_{A\calpha} &:= (-\kcsq D+\A)^{-1} (\langle \M \E_0, \E_0^*\rangle- \M) \E_0,\\
w_{A\beta} &:= \kc(-\kcsq D+\A)^{-1}(\langle \B\E_0,\E_0^*\rangle -\B)\E_0,\\
w_{A\kap} &:= 2\kc(-\kcsq D+\A)^{-1} D \E_0,\\
w_{A\beta\beta}&:= 2\kc(-\kcsq D+\A)^{-1} (\B w_{A\beta} - \langle \B w_{A\beta},E_0^*\rangle\E_0),\\
e_\mu(x)&:=(\E_0+\calpha  w_{A\calpha} + \rmi\beta w_{A\beta} + \kap w_{A\kap} + \beta^2 w_{A\beta\beta}) \rme^{\rmi x}.
\end{aligned}
\end{equation}
We note that the evaluation at $\mu=0$ in the following theorem gives the velocity parameter $c=-\lambda_\beta$ and at this value of $c$ we have $\langle \B\E_0,\E_0^*\rangle =0$.

\begin{theorem}[Stripe existence]\label{t:bif}
Up to spatial translation, non-trivial stripe solutions to \eqref{e:RDS} with parameters $\mu$, and sufficiently small $|\mu|, A$ with $\|U_\rms(\cdot;\mu)\|_{\Lspace^2}=\calO(A)$ on $[0,2\pi/\kappa]$, are in 1-to-1 correspondence with solutions $A> 0$ to
\begin{equation}\label{e:stripeeqn}
\Re(\tlam(\mu)) + \rho_\nl A^2 + \calO\left(A^3\right)= 0,
\end{equation}
where $\tlam(\mu) =  r(\mu)\lambda_\mu$, cf.\ \eqref{e:evlinearop}, and $r$ is smooth with $r(0)=1$. 
Stripes have velocity $\beta c$ with 
\begin{align}\label{e:stripev}
c = -\lambda_{\beta}-\frac{\lambda_{\M\beta}}{\kc}a_\M\alpha - \frac{\lambda_{\kap\beta} - \lambda_{\beta}}{\kc}\kap 
+ \calO(a_\M|\alpha\kap|+\kap^2 + |A|^3)
\end{align}
and, in this comoving frame, are of the form 
\begin{align}\label{e:Stripes}
U_\rms(x;\mu) =\ & A(e_\mu(x) + \overline{e_\mu(x)}) + \frac{A^2}{2} \hQ\left( \rme^{2 \rmi x} +  \rme^{-2 \rmi x}\right)+ A^2\tQ + \calR,
\end{align}
with the smooth remainder $\calR=\calO(|A|(A^2 + a_M\alpha^2 + \kap^2 + |\beta\kap| + |\beta|^3))$ near $\mu=0$.  
Moreover, the coefficients in the expansion of $\tlam$ analogous to \eqref{e:evlinearop} satisfy 
\begin{equation}\label{e:stripecoeffalt}
\begin{aligned}
\lambda_\M &= \langle M\E_0, \E_0^* \rangle,\quad 
\lambda_{\M\beta} = \langle M w_{A\beta} + \kc \B w_{A\calpha}, \E_0^* \rangle/\lambda_\M,\\
\lambda_{\M\kap} &= \langle M w_{A\kap} - 2\kc D w_{A\calpha}, \E_0^* \rangle/\lambda_\M,\\
\rho_\beta &=  -\kc\langle \B w_{A\beta}, \E_0^*\rangle, \quad
\rho_\kap = -2\kc\langle D w_{A\kap}, \E_0^* \rangle,\\
\gamma_\beta &= \kc\langle \B\E_0, \E_0^* \rangle, \quad 
\gamma_{\kap\beta} = \kc\langle \B w_{A\kap} - 2D w_{A\beta}, \E_0^* \rangle + \langle \B\E_0,\E_0^* \rangle.
\end{aligned}
\end{equation}
\end{theorem}

We defer the proof to Appendix \ref{s:bifproof}. In case $M=\rm{Id}$ clearly $\alpha$ uniformly shifts spectra so that $\alpha$ does not impact higher orders in $\lambda_\mu$ as can be seen from the fact that $\lambda_\M=1$ and $w_{A\calpha}=\lambda_{\M\beta}=\lambda_{\M\kap}=0$ in this case.

In its simplest case, the theorem reflects the well-known fact that, up to translation symmetry, for $\rho_\nl\neq 0$ the bifurcation is a generic pitchfork. Specifically, with respect to $\alpha$ the bifurcation is supercritical if $\rho_\nl<0$, which we shall assume  in the following stability study. 
 
Our main interest lies in the role of $\beta$ and $Q$. As noted in the discussion of the eigenvalues above, $\beta$ shifts the bifurcation points by order $\beta^2$. From \eqref{e:stripeeqn} we readily solve for the stripe amplitude as
\begin{align}\label{e:ampfull}
A = \left(1+ \calO\left(\sqrt{\Re(\tlam(\mu))}\right)\right)\sqrt{-\frac{\Re(\tlam(\mu))}{\rho_\nl}}.
\end{align}

Notably, the bifurcation loci, where $A=0$, occur on a surface in $\mu$-space that includes $\mu=0$ since the signs of $\rho_\beta$ and $\rho_\kap$ are opposite. The leading order part $\alpha + \rho_\beta\beta^2 + \rho_\kap\kap^2$ of $\Re(\tlam(\mu))$ coincides with that of $\Re(\lambda_\mu)$ and is homogeneous with respect to the scalings $\alpha=A^2\alpha'$, $\beta=A\beta'$, $\kap=A\kap'$, for which $A=0$ occurs at $\mu'=(\alpha',\beta',\kap')=0$ only. In these scaled parameters with $\mu'=\calO(1)$ the order analysis of remainders drastically simplifies to $\calR=\calO(A^3)$. In particular, the scaling to $\mu'$ gives
\begin{align}\label{e:amp}
A = A\sqrt{-\frac{\alpha' + \rho_\beta\beta'^2 + \rho_\kap\kap'^2}{\rho_\nl}}+ \calO(A^2).
\end{align}

\medskip
\begin{remark}\label{r:direction}
The sign of $c=c(\mu)$ is the direction of stripe motion relative to $\beta$, and is determined by $\lambda_{\beta}$ as $\sgn(c) = -\sgn(\lambda_{\beta})$. In terms of $a_1,a_4$  we have 
\[ 
\sgn(c) = -\sgn(a_1) 
\]
so the motion is with $\beta$ if the first component is an inhibitor and against $\beta$ otherwise.
\end{remark}
\begin{proof}
Recall the notation of Lemma~\ref{l:Turbeta} and Lemma~\ref{l:lambetamata}, which gives $\lambda_{\beta} = \frac{b_4}{b_1+b_4}$, and we have $b_1b_4=b_2b_3<0$ and $b_1+b_4<0$. 
For case (1) we note $a_1<0$ implies $b_1<0$, which implies $b_4>0$ and thus the claim.
For case(2) similarly from $a_4<0$, we have $b_4<0$ which leads to $b_1>0$. Hence $\lambda_{\beta} = \frac{b_1+b_4-b_1}{b_1+b_4} = 1-\frac{b_1}{b_1+b_4}>1$ implies $c<-1$.
\end{proof}

\section{Large wavelength stability}\label{s:stability}

Linearising \eqref{e:RDSper} in a stripe solution gives the operator and eigenvalue problem for a perturbation $U$, 
\begin{align}\label{e:evalstripe}
\calL_\mu U + 2\Q[U_\rms,U] + 3\K[U_\rms,U_\rms,U] = \lambda U,
\end{align}
e.g.~in the function space setting noted in Appendix~\ref{s:bifproof}. It is convenient to write the stripes in real terms, 
\begin{align*}
	U_\rms(x;\mu) =\ &2A(\E_0+ \kap w_{A\kap} + \calpha w_{A\calpha} + \beta^2 w_{A\beta\beta})\cos(x) - 2A\beta w_{A\beta} \sin(x)\\
	 &+ A^2\hQ\cos(2x)+ A^2\tQ + \calR.
\end{align*}
As we now view stripes $U_\rms(x)$  in two space dimensions $\x=(x,y)\in\R^2$ we may Fourier-transform \eqref{e:evalstripe} with respect to $y$ thus replacing $\partial_y^2$ by $-\ell^2$. In $x$ we perform a Floquet-Bloch-transform, i.e., in $\calL_\mu=\calL_\mu(\partial_x)$ replace $\partial_x$ by $\partial_x+\rmi\gamma$ and impose periodic boundary conditions on $[0,2\pi]$, e.g.,~\cite{Rademacher2007}. From \eqref{e:evalstripe} this gives the operator
\begin{align*}
\calT &:= \kappa^2 D ((\partial_x+\rmi\gamma)^2-\ell^2) + \A + \calpha \M + \beta\kappa\B (\partial_x + \rmi\gamma)  + 2\Q[U_\rms,\cdot] + 3\K[U_\rms,U_\rms,\cdot],
\end{align*}
which, as usual, arises for the perturbation in the form
\[
U(\x) = \rme^{\rmi(\gamma x + \ell y)}V(x;\gamma),
\]
where $V(x;\gamma)$ has periodicity of $U_\rms(x)$ in $x$ and we write $V_0(x):=V(x;0)\in\R^2$.

\medskip
Here we are interesting in the stability of stripes against large wavelength perturbations, i.e., $\gamma,\ell\approx 0$. 
Let us consider the eigenvalue problem $\calT V= \lambda V$ with respect to a parameter $p\in\{\ell,\gamma\}$ and denote evaluations at $p=0$ by subindex zero. The curve of eigenvalues attached to the translation mode at the origin thus has $\lambda|_0=0$, which is a simple zero eigenvalue with eigenvector $V_0$. The kernel of $\calT_0$ is therefore spanned by
\begin{align*}
\partial_xU_\rms =\ & -2A(\E_0 + \kap w_{A\kap}+ \calpha w_{A\calpha} + \beta^2 w_{A\beta\beta})\sin(x) - 2A\beta w_{A\beta}\cos(x)\\
& - 2A^2\hQ\sin(2x) + \calO(\calR).
\end{align*}  

Differentiating $\calT V= \lambda V$ with respect to $p$ and evaluating at $p=0$ gives
\begin{align}\label{e:firstdiff}
\calT_0(\partial_p V)_0  = (\partial_p\lambda)_0 V_0 - (\partial_p \calT)_0 V_0.
\end{align}
By Fredholm alternative, this is solvable in $(\partial_p V)_0$ if and only if the right-hand side is orthogonal to the kernel of adjoint operator of $\calT_0$ and thus
\begin{align}\label{e:firstdif}
	(\partial_p \lambda)_0 = \langle (\partial_p\calT)_0 V_0, V_0^* \rangle
\end{align}
with the normalisation $\langle V_0,V_0^* \rangle = 1$ and $V_0^*$ in the kernel of the adjoint operator
\[
\calT_0^* := \kappa^2 D \partial_x^2 + \A^T + \calpha \M^T - \beta\kappa\B \partial_x  + (2\Q[U_\rms,\cdot] + 3\K[U_\rms,U_\rms,\cdot])^T.
\]
In case $p=\ell$ we have $(\partial_\ell \calT)_0=0$ and it follows that $(\partial_\ell\lambda)_0=0$.
In case $p=\gamma$,
\begin{align}\label{e:Tgam}
(\partial_\gamma\calT)_0 = 2\rmi\kappa^2D\partial_x + \rmi\beta\kappa\B,
\end{align}
and it follows that
\begin{align}\label{e:gamfirstdiff}
	(\partial_\gamma\lambda)_0 = \rmi\kappa\langle (2\kappa D\partial_x + \beta\B)V_0,V_0^* \rangle \in \rmi\R,
\end{align}
which measures the correction of the phase velocity $c$ to the group velocity, cf. \cite{Doelman2009} and Remark \ref{r:groupv} in Appendix~\ref{s:Eckhausproof}. It is well known to vanish for stationary wavetrains $c=0$ due to reflection symmetry in $x$ of $V_0=\partial_x U_\rms$ and $V_0^*$; here this occurs for $\beta=0$ so that $(\partial_\gamma\lambda)_0=\calO(|\beta|)$ as we shall confirm in Appendix~\ref{s:Eckhausproof}.

Differentiating again and evaluating at $p=0$ gives 
\[
\calT_0(\partial_p^2 V)_0 = (\partial_p^2\lambda)_0 V_0 - (\partial_p^2\calT)_0 V_0 + 2(\partial_p\lambda)_0 (\partial_p V)_0 - 2(\partial_p\calT)_0 (\partial_p V)_0.
\]
Proceeding as above, in case $p=\ell$ we have
\begin{align}\label{e:seconddifell}
(\partial_\ell^2\lambda)_0 = \langle (\partial_\ell^2\calT)_0V_0,V_0^* \rangle = -2\kappa^2\langle DV_0,V_0^* \rangle,
\end{align}
and in case $p=\gamma$ we have
\begin{align}\label{e:seconddigfam}
(\partial_\gamma^2\lambda)_0 = \langle (\partial^2_\gamma\calT)_0V_0 - 2(\partial_\gamma\lambda)_0 (\partial_\gamma V)_0 + 2(\partial_\gamma\calT)_0(\partial_\gamma V)_0, V_0^* \rangle.
\end{align}
These quantities give the curvatures of spectrum at the origin in $\ell$ and $\gamma$ directions, respectively. Other directional derivatives are not relevant since $(\partial_\ell V)_0\in\ker\calT_0$ and thus $(\partial_{\ell\gamma}\lambda)_0=0$. Hence, the signs of \eqref{e:seconddifell}, \eqref{e:seconddigfam} determine the sideband stability or instability of stripes, which is commonly referred to as  Eckhaus un/stable for the $x$-direction, i.e. with respect to $\gamma$ and $\ell=0$, and as zigzag un/stable for the $y$-direction, i.e., with respect to $\ell$ and $\gamma=0$.

\bigskip
\paragraph{Zigzag instability} It is well-known that stripes become unstable against large wavelength perturbations parallel to the stripes if the stripes are stretched, while stripes are not as sensitive to compression. The canonical example is the Swift-Hohenberg equation which has not advection or quadratic terms. The main point of the next theorem is to exhibit the effect of advection through $\beta$ and also the role of  quadratic terms in the system.

\begin{theorem}[Zigzag instability]\label{t:zigzag}
For $\mu$ such that the stripe solution \eqref{e:Stripes} with the amplitude $A(\mu)>0$ exists in \eqref{e:RDS}, the curve of spectrum of $\calT$ for $\gamma=0$ and $|\ell|\ll1$ attached to the origin is given by
\begin{align}\label{e:speczz}
	\lambda_\zz(\ell) = \left(\kc\rho_\kap\kap + \rho_\calpha\alpha + \rho_{\bbeta}\beta^2 +\calR_\zz\right)\ell^2,
\end{align}
with $\rho_\kap$ as in \S\ref{s:bif}, and
\begin{equation}\label{e:coeffzz}
	\begin{aligned}
	\rho_\calpha &:= -a_\M\kcsq(\langle D\E_0,w_{A\calpha}^*\rangle +\langle D w_{A\calpha},\E_0^*\rangle)/\lambda_\M - q_{22}/\rho_\nl,\\
	\rho_{\bbeta} &:= -\kcsq(\langle D\E_0,w_{A\beta\beta}^*\rangle +\langle Dw_{A\beta\beta},E_0^*\rangle-\langle Dw_{A\beta}, w_{A\beta}^* \rangle) - q_{22}\rho_\beta/\rho_\nl,\\
	q_{22} &:= -\kcsq\langle D\Q_2,\Q_2^*\rangle, \quad\calR_\zz=\calO(a_\M\alpha^2+\kap^2+a_\M|\alpha|\beta^2+\ell^2),
	\end{aligned}
\end{equation}
	where $a_\M=0$ if $M={\rm Id}$, $a_\M=1$ otherwise.
\end{theorem}

The proof is presented in Appendix \ref{s:zigzagproof}. The theorem in particular shows that $\lambda_\zz$ depends to quadratic order on the advection parameter $\beta$. In particular, for $\rho_\calpha\neq 0$, the theorem gives the zigzag stability boundary to leading order as 
\begin{align}\label{e:zzbnd}
\alpha = \calZ(\kap,\beta) = -(\kc\rho_\kap\kap + \rho_\bbeta\beta^2)/\rho_\calpha.
\end{align}

Recall that $A=0$ holds for a surface in $\mu$-space that includes $\mu=0$. The natural scalings discussed after \eqref{e:ampfull} give the following reduced spectrum.
\begin{corollary}\label{c:zzscaling}
Assume the conditions in Theorem \ref{t:zigzag} and the scalings $\alpha=A^2\alpha'$, $\beta=A\beta'$, $\kap=A\kap'$, then the curve of spectrum of $\calT$ for $\gamma=0$ and $|\ell|\ll1$ attached to the origin is given by
\[
\lambda_\zz(\ell) = A\left(\kc\rho_\kap\kap' + \calO(|A|)\right)\ell^2.
\]
\end{corollary}

Here the zigzag stability boundary is given by $\kap=0$ to leading order, independent of the advection, cf. Fig.~\ref{f:zzbet0}\,\&\,\ref{f:zzbetn0ralp0} (green lines).
\begin{figure}[h!]
	\centering
	\subfloat[$\beta=0$, $\rho_\calpha\neq0$]{\includegraphics[width=0.25\linewidth]{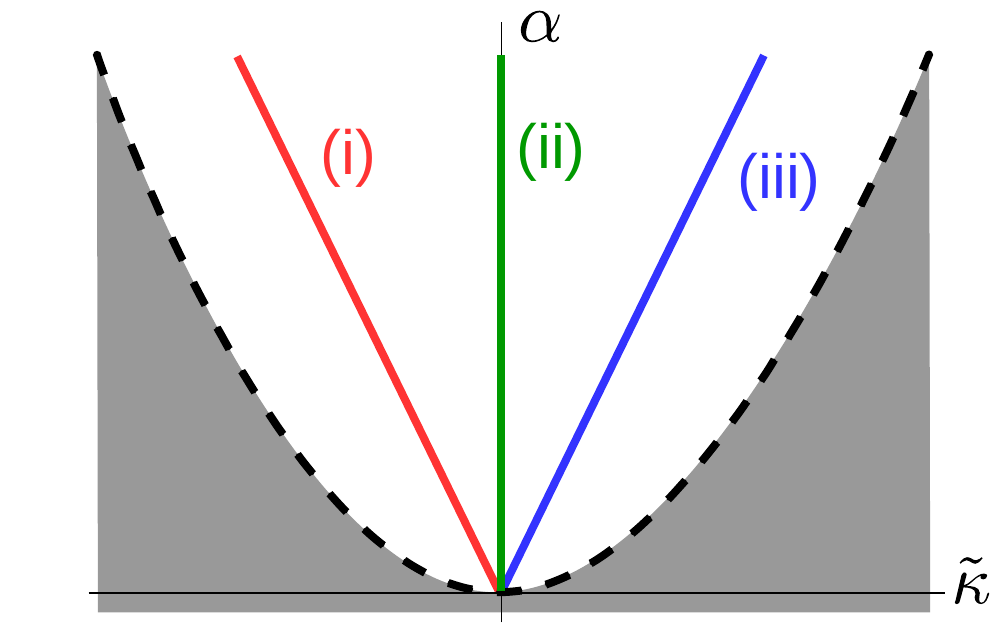}\label{f:zzbet0}}
	\hfil
	\subfloat[$\beta\neq0$, $\rho_\calpha=0$]{\includegraphics[width=0.25\linewidth]{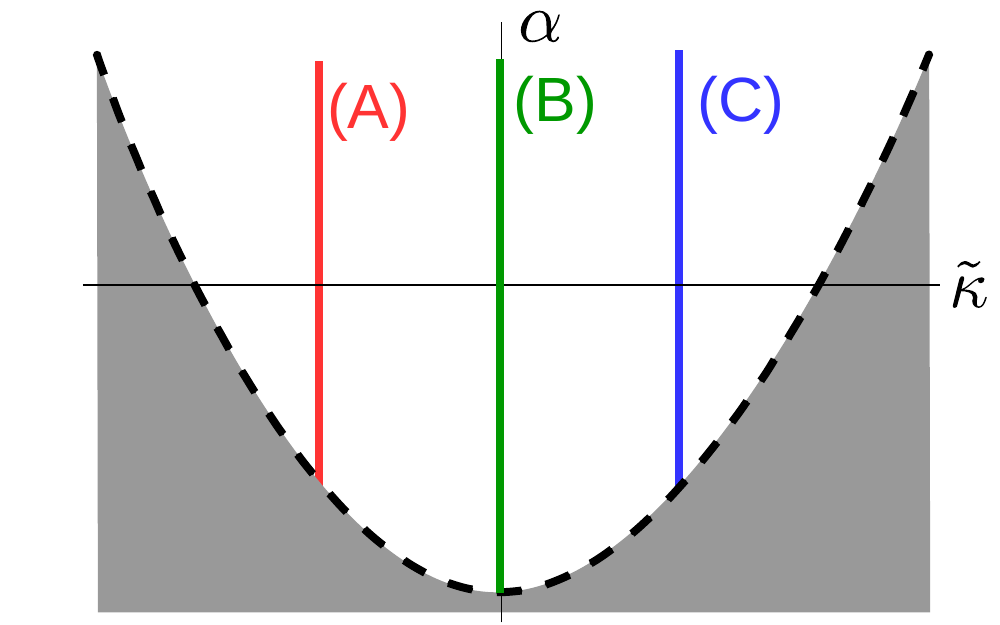}\label{f:zzbetn0ralp0}}
	\hfil
	\subfloat[$\beta\neq0$, $\rho_\calpha<0$]{\includegraphics[width=0.25\linewidth]{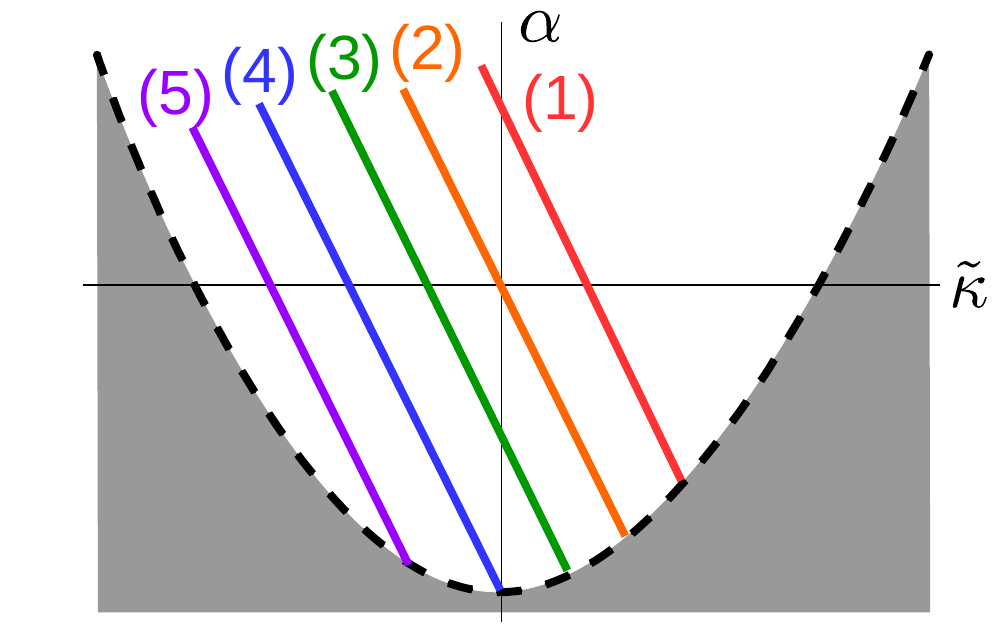}\label{f:zzn}}
	\hfil
	\subfloat[$\beta\neq0$, $\rho_\calpha>0$]{\includegraphics[width=0.25\linewidth]{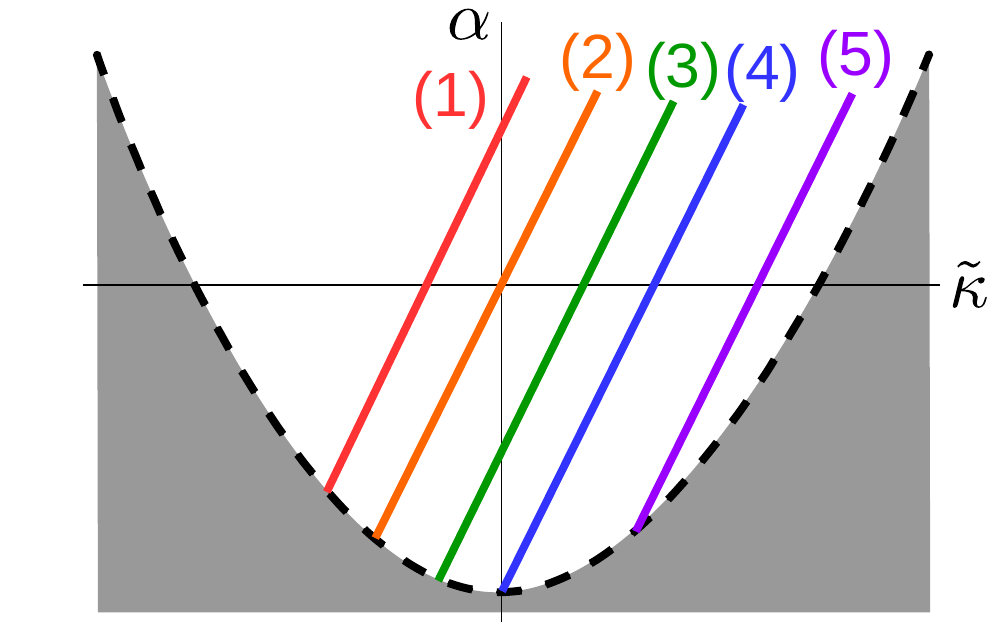}\label{f:zzp}}
	\caption{Sketches of the different leading order zigzag boundaries in the $(\kap,\alpha)$-plane. Stripes exist in the white regions. Dashed curves: bifurcation curves; coloured solid lines: zigzag boundaries. The zigzag unstable region lies to the left of each zigzag boundary. In (a): (i) $\rho_\calpha<0$, (ii) $\rho_\calpha=0$, zigzag boundary is $\kap=0$, (iii) $\rho_\calpha>0$. In (b): (A) $\rho_\bbeta<0$, (B) $\rho_\bbeta=0$, zigzag boundary is $\kap=0$, (C) $\rho_\bbeta>0$. In (c) and (d): (1) $\calZ(0,\beta)>0$ ($\rho_\bbeta/\rho_\calpha<0$), (2) $\calZ(0,\beta)=0$ ($\rho_\bbeta=0$), (3) $\calB(0,\beta)<\calZ(0,\beta)<0$ ($0<\rho_\bbeta/\rho_\calpha<\rho_\beta$), (4) $\calZ(0,\beta)=\calB(0,\beta)$ ($\rho_\bbeta/\rho_\calpha=\rho_\beta$), (5) $\calZ(0,\beta)<\calB(0,\beta)$ ($\rho_\bbeta/\rho_\calpha>\rho_\beta$).}
	\label{f:zzfull}
\end{figure}

Relaxing these scalings assumptions yields a variety of zigzag stability boundaries depending on the signs of $\rho_\calpha$ and $\rho_\bbeta$, cf. Fig.~\ref{f:zzfull}. Nonzero $\rho_\calpha$ creates a sloping zigzag boundary and nonzero $\rho_\bbeta$ shifts the zigzag boundary horizontally. As mentioned in Fig.~\ref{f:sidebandintrobetn0}, the attachment point of the zigzag boundary to the bifurcation loci can be moved and rotated relative to $\kap=0$. The bifurcation curve at $\kap=0$ lies at
\[
\alpha=\calB(0,\beta) = -\rho_\beta\beta^2,
\] 
and the zigzag boundary at $\kap=0$ lies at
\[
\alpha =\calZ(0,\beta)=-\frac{\rho_\bbeta}{\rho_\calpha}\beta^2
\]
for $\rho_\calpha\neq0$. Hence we can compare $\calB(0,\beta)$ and $\calZ(0,\beta)$ and obtain the more accurate positions of the zigzag boundaries near the bifurcation curve and close to $\kap=0$, cf. Fig.~\ref{f:zzfull}.

Notably, the term $q_{22}$ related to the quadratic form $\Q$ appears in both $\rho_\calpha$ and $\rho_\bbeta$. In particular, vanishing quadratic form $\Q=0$ gives $q_{22}=0$.

\begin{remark}\label{r:rbb}
For $Q=0$ we have $q_{22}=0$, and from Remark~\ref{r:ev}, as well as \eqref{e:defs} and \eqref{e:stripev} a tedious computation gives
\begin{align*}
\rho_{\bbeta}= \frac{{\bf k}_{\rm c}^4b_3^2d_1}{b_1^2(b_1+b_4)^4}b_4(5b_1+b_4).
\end{align*}
Recall $b_1b_4 =b_2b_3<0$, $b_1+b_4<0$ and $a_1<0$ implies $b_1<0$. Hence, for all sufficiently small coefficients in $\Q$, we have $\rho_{\bbeta}<0$ for either $a_1<0$, or $a_1>0$ and $a_1>\kcsq d_1-(a_4-\kcsq d_2)/5$, $\rho_{\bbeta}>0$ otherwise.
\end{remark}

\medskip
From Theorem~\ref{t:bif} we know that the bifurcation curve  for $\kap=0$ in the $(\beta,\alpha)$-plane is to leading order given by $\alpha=-\rho_\beta\beta^2$. In order to study the stability at the onset of bifurcation, let us consider $\talpha:=\alpha+\rho_\beta\beta^2$ so the bifurcations occur at $\talpha=0$ in the $(\beta,\talpha)$-plane. It follows that
\begin{align}
\lambda_\zz(\ell) &= \left(\kc\rho_\kap\kap + \rho_\calpha(\talpha - \rho_\beta\beta^2) + \rho_{\bbeta}\beta^2 +\calR_\zz\right)\ell^2\nonumber\\
&= \left(\kc\rho_\kap\kap + \rho_\calpha\talpha + (\rho_{\bbeta}-\rho_\calpha\rho_\beta)\beta^2 +\calR_\zz\right)\ell^2,\label{e:speczz2}
\end{align}
and thus the zigzag boundary for $\rho_\calpha\neq0$ is given by
\begin{align}\label{e:zzbndtalp}
\talpha = -(\kc\rho_\kap\kap + (\rho_\bbeta-\rho_\calpha\rho_\beta)\beta^2)/\rho_\calpha.
\end{align}
This dependence on $\beta$ shows that the advection influences the form of the zigzag stability boundary near the bifurcation.

\bigskip
\paragraph{Eckhaus instability} It is well known that a supercritical Turing bifurcation for $\beta=\kap=0$ implies stable Eckhaus sideband, and we next determine the expansion including $\beta,\kap$. Recall the Eckhaus instability arises from perturbations that vary only in $x$-direction, i.e., $p=\gamma$ and $\ell=0$. 

\begin{theorem}[Eckhaus instability]\label{t:Eckhaus}
For $\mu$ such that the stripe solution \eqref{e:Stripes} with amplitude $A(\mu)>0$ exists in \eqref{e:RDS}, the curve of spectrum of $\calT$ for $\ell=0$ and $|\gamma|\ll1$ attached to the origin is given by
\begin{align*}
	\lambda_\eh =\rmi\kc\left((\lambda_{\kap\beta}-\lambda_{\beta})\beta + \calO(A^2)\right)\gamma -\kcsq\frac{\rho_\kap}{\rho_\nl}A^{-2}\left(\alpha+\rho_\beta\beta^2+3\rho_\kap\kap^2
	+\calR_\eh)\right)\gamma^2,
\end{align*}
with $\calR_\eh:=\calO(\mu^2 + A^2|\mu| + A^4 + |\gamma|)$.
\end{theorem}

See Appendix \ref{s:Eckhausproof} for the proof and revisit Fig.~\ref{f:sidebandintro} for the (un)stable regions. Here we have simplified the estimate of $\calR_\eh$ -- more details can be found in the proof.

Hence, the Eckhaus stability boundary is given to leading order by 
\begin{align}\label{e:Eckhausbnd}
	 \alpha = \calE(\kap,\beta) = - 3\rho_\kap\kap^2-\rho_\beta\beta^2.
\end{align}
We note that for $\kap=0$ this is $\calE(0,\beta) =-\rho_\beta\beta^2 = \calB(0,\beta)$. Moreover, since $\rho_\kap<0$, we have  $\calE(\kap,\beta) \geq \calB(\kap,\beta)$ so that, as usual, the Eckhaus boundary touches the bifurcation curve at $\kap=0$ and lies in the existence region of stripes. Therefore, for $\kap=0$ the bifurcating stripes are Eckhaus stable and unstable otherwise. 

Analogous to the zigzag stability, we consider $\talpha:=\alpha+\rho_\beta\beta^2$ so that 
\begin{align}
\Re(\lambda_\eh) =-\kcsq\frac{\rho_\kap}{\rho_\nl}A^{-2}\left(\talpha+3\rho_\kap\kap^2+\calR_\eh)\right)\gamma^2,
\end{align}
and the Eckhaus boundary becomes
\begin{align}\label{e:ehbndtalp}
\talpha = -3\rho_\kap\kap^2,
\end{align}
which is independent on $\beta$ to leading order -- in contrast to the zigzag boundary. Hence, the leading order effect of advection through $\beta$ is just a translation of the Eckhaus boundary downwards ($\rho_\beta>0$) with order $\beta^2$. In other words, for any fixed $\alpha$ in the existence region, the width of Eckhaus stable region increases with $|\beta|$. The advection can well influence the Eckhaus stability at higher order, cf. \eqref{e:ehhot}, but an analysis of this is beyond the scope of this paper.

\section{Examples}\label{s:example}

\subsection{Exact example: zigzag-unstable stripes}

For illustration of the expansions we consider the concrete system
\begin{equation}\label{e:example}
	\begin{aligned}
	u_t &= \Delta u + 3u - v + \calpha u + 4\calpha v + \beta u_x + \frac{1}{2}u^2+\frac{1}{8}v^2- uv^2\\
	v_t &= \frac{7}{2}\Delta v + 14u - \frac{7}{2}v  -\frac{1}{5}\calpha u +\calpha v + \frac{1}{2}u^2+\frac{1}{8}v^2 + uv^2
	\end{aligned}
\end{equation}
where $U:= (u,v)^T$, $D = \diag(1,7/2)$,
\begin{align*}
	\A = \begin{pmatrix} 3 & -1\\ 14 & -\frac{7}{2} \end{pmatrix}\!,\; \M = \begin{pmatrix} 1 & 4\\ -\frac{1}{5} & 1 \end{pmatrix}\!,\; \Q[U,U] = \frac{1}{2}\begin{pmatrix} u^2+\frac{1}{4}v^2 \\ u^2+\frac{1}{4}v^2 \end{pmatrix}\!,\; \K[U,U,U] = \begin{pmatrix} -uv^2 \\ uv^2 \end{pmatrix}\!.
\end{align*}
The generic form of $\Q$ is given by $\Q[U_1,U_2] = (\Q_|[U_1,U_2],\Q_{||}[U_1,U_2])^T$ with
\[
\Q_|[U_1,U_2] = \Q_{||}[U_1,U_2] = U_1^T \begin{pmatrix}
\frac{1}{2} & 0 \\ 0 & \frac{1}{8}
\end{pmatrix} U_2,
\]
where $U_j:= (u_j,v_j)^T$, $j = 1,2,3$. 

In this system, the Turing conditions are fulfilled and the critical wavevectors $(k,\ell)\in S_\kc$ with $\kc = 1$. We have
\begin{align*}
	\hat\calL_0 := -\kcsq D+\A = 
	\begin{pmatrix}
	2 & -1\\
	14 & -7
	\end{pmatrix}.
\end{align*}
From Remark~\ref{r:ev} the rescaled kernel eigenvector of $\hat\calL_0$ and its adjoint kernel eigenvector are given by
\begin{align*}
	\E_0 = -\frac{1}{\sqrt{5}}(1,2)^T,\quad 
	\E_0^* = \frac{1}{\sqrt{5}}(-7,1)^T.
\end{align*}
We examine the coefficients in \eqref{e:defs}, \eqref{e:stripecoeffalt} and \eqref{e:coeffzz} so that they are nonzero. The bifurcation curves, zigzag and Eckhaus boundaries are given by, cf.\ Fig.~\ref{f:example},
\begin{align}
\text{bifurcation curve:}\quad & \alpha = -\frac{14}{125}\beta^2+\frac{14}{5}\kap^2,\label{e:egbif}\\
\text{Eckhaus boundary:}\quad & \alpha=-\frac{14}{125}\beta^2+\frac{42}{5}\kap^2,\label{e:egeh}\\
\text{zigzag boundary:}\quad & \alpha = -\frac{952}{267125}(13\beta^2 + 875\kap).\label{e:egzz}
\end{align}

The striped solutions exist for $\alpha>-\frac{14}{125}\beta^2+\frac{14}{5}\kap^2$. In Fig.~\ref{f:stripe} we plot the leading order form of a stripe based on \eqref{e:Stripes} for $\alpha=0.2$, $\beta=0.7$ and $\kap=0.1$, which gives the velocity parameter $c=-7/5$.
\begin{figure}[h!]
\centering
\includegraphics[width=0.45\linewidth]{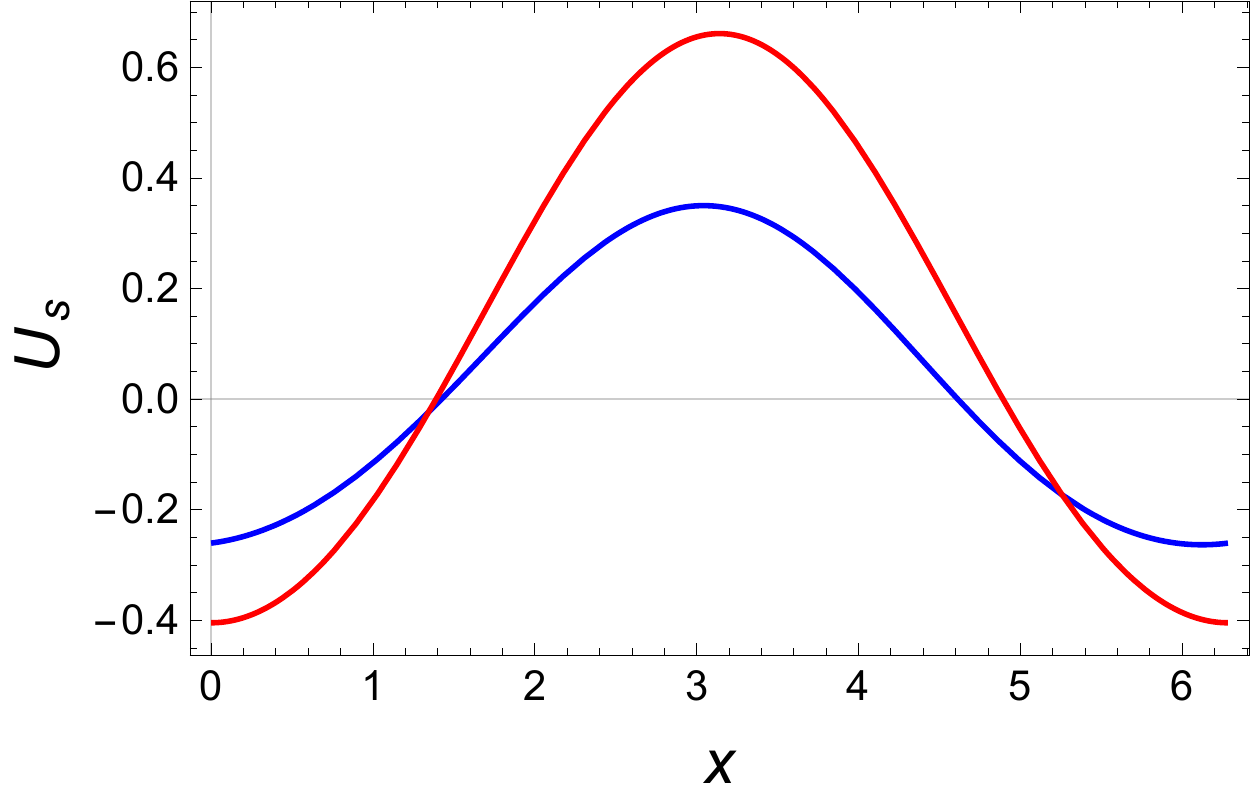}
\caption{The leading order of the rescaled striped solution $U_\rms$ ($u$-component blue, $v$-component red) in $x\in[0,2\pi]$ to the system \eqref{e:RDS} for $\alpha=0.2$, $\beta=0.7$ and $\kap=0.1$.}
\label{f:stripe}
\end{figure}

The advection term shifts the bifurcation curve and the Eckhaus boundary downwards since the coefficient of $\beta^2$ are both negative in \eqref{e:egbif} and \eqref{e:egeh}, cf. Fig.~\ref{f:example}. Thus the advection stabilises the large wavelength perturbations in the $x$-direction. 

The negative coefficient of $\kap$ in \eqref{e:egzz} adds a negative value to the slope of the zigzag boundary, cf. Fig.~\ref{f:example}. The negative coefficient of $\beta^2$ shifts the zigzag boundary to the left, cf. Fig.~\ref{f:sbbet0}\,\&\,\ref{f:sbbetn0}. Hence the advection stabilises the large wavelength perturbations in the $y$-direction for any $\alpha>0$. Since the coefficient of $\beta^2$ in \eqref{e:egzz} is larger than that of \eqref{e:egbif}, however, there exists a zigzag unstable region near the bifurcation curve and for $\kap>0$, cf.\ Fig.~\ref{f:zzn}. We plot the resulting curves in Fig.~\ref{f:sbbetn0}. In particular, the width of this unstable region is of order $\beta^2$. Hence the advection destabilises the large wavelength perturbations in the $y$-direction at the onset of Turing bifurcation. This can also be seen from the positive coefficient of $\beta^2$ in \eqref{e:speczz2}.

\begin{figure}[h!]
\centering
\subfloat[$\beta=0$]{\includegraphics[width=0.45\linewidth]{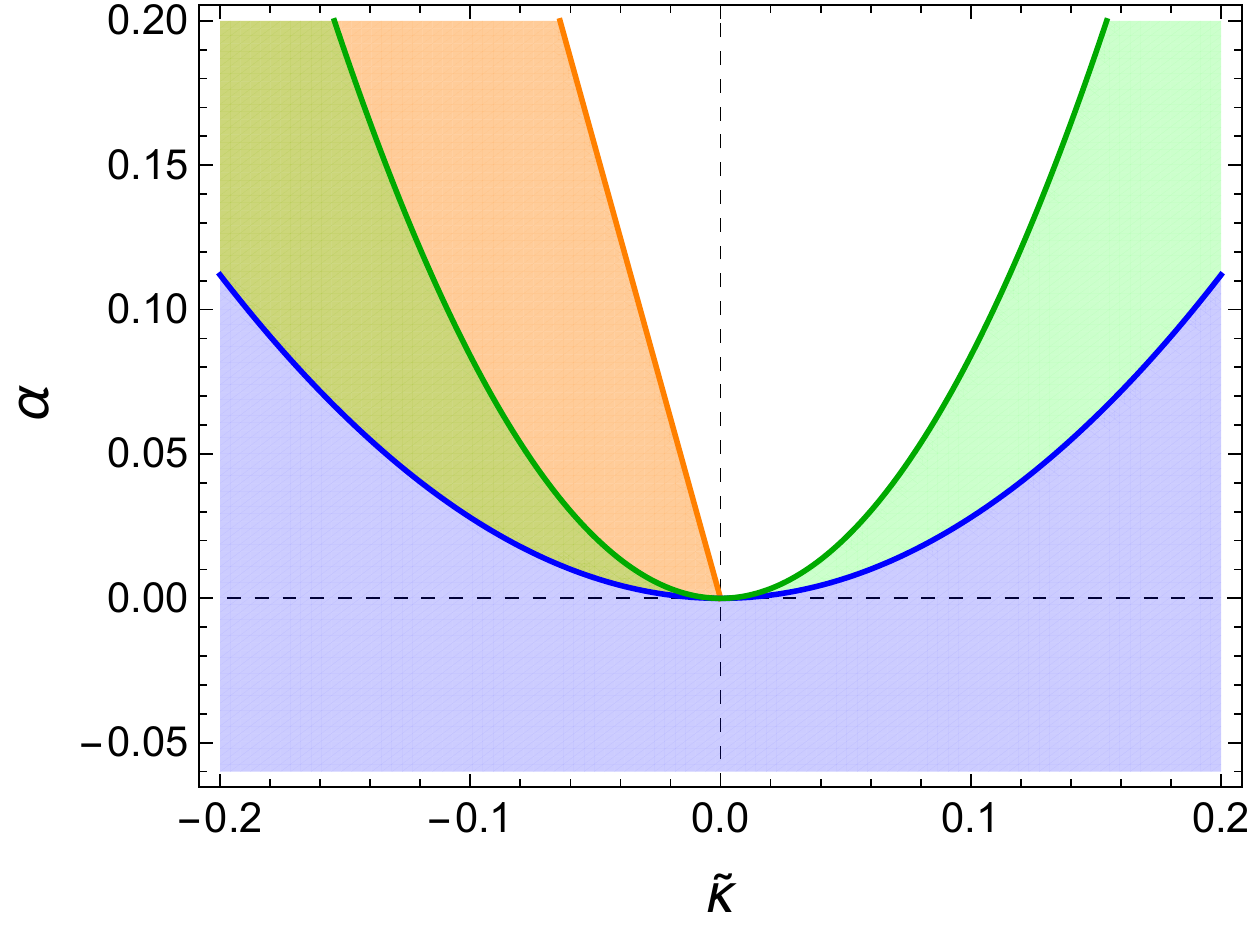}\label{f:sbbet0}}
\hfil
\subfloat[$\beta=0.7$]{\includegraphics[width=0.45\linewidth]{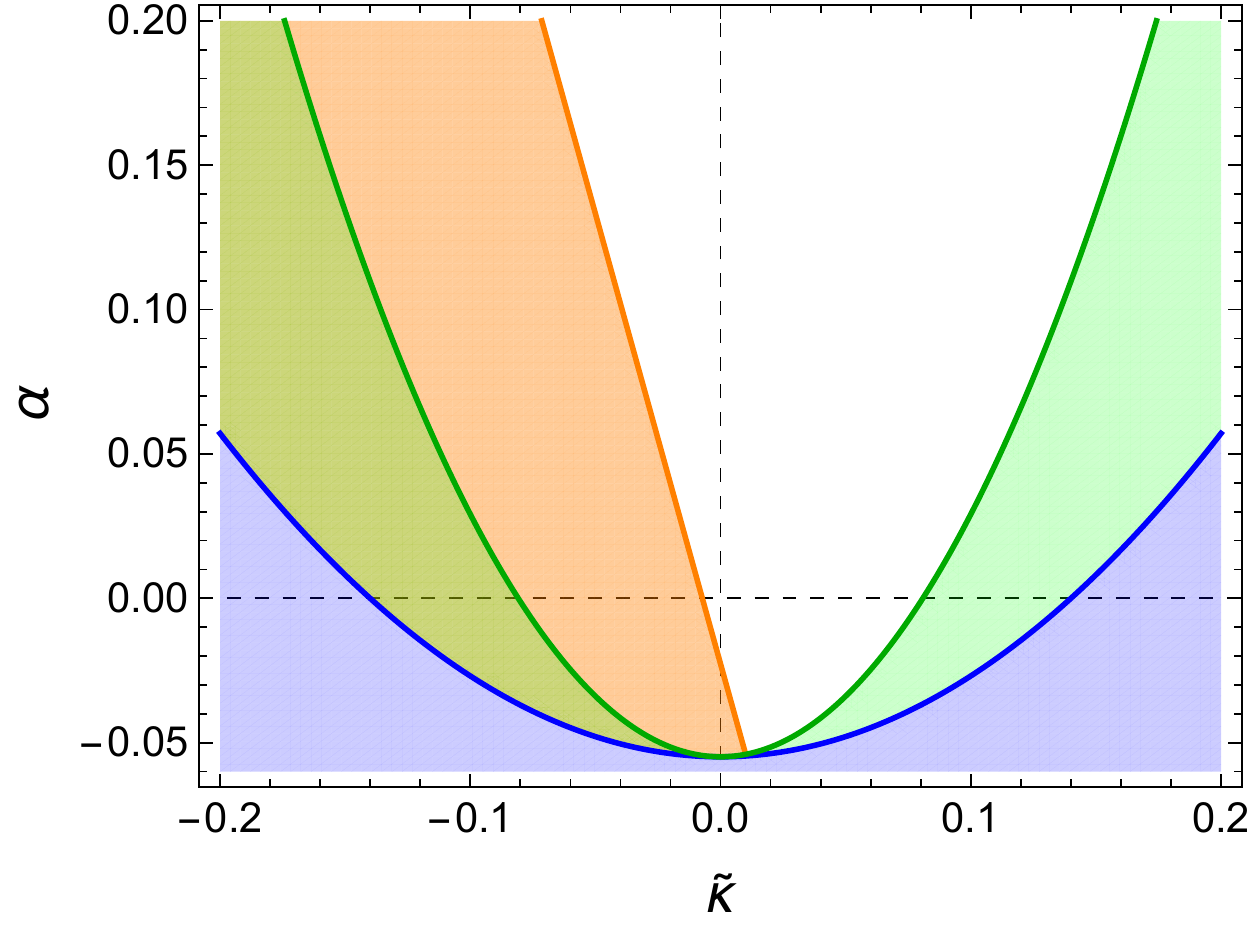}\label{f:sbbetn0}}
\caption{Numerical computations of the leading order Eckhaus and zigzag (in)stability regions of the stripes for \eqref{e:example} in the $(\kap,\alpha)$-plane. Stripes exist in the complement of the blue regions. Blue lines: bifurcation curves \eqref{e:egbif}; green regions: Eckhaus unstable; green lines: Eckhaus boundaries \eqref{e:egeh}; orange regions: zigzag unstable; orange lines: zigzag boundaries \eqref{e:egzz}; white regions: stable stripes. (a) $\beta=0$. (b) $\beta=0.7$. In (a) the zigzag boundary is attached to the origin, whereas in (b) the origin is stable, but advection shifts attachment point of the zigzag boundary to the right; $\M\neq\Id$ destabilises the stripes near $\kap=0$.}
\label{f:example}
\end{figure}

\subsection{Numerical example: extended Klausmeier model}\label{s:Klausmeier}

The extended Klausmeier in two space dimensions \cite{Klausmeier1999,Siero2015} is a two-component model for studying vegetation patterns on the earth's surface in drylands. In scaled form it is given by:
\begin{equation}\label{e:Klausmeier}
       \begin{aligned}
       u_t &= d\Delta u + \beta u_x+ a - u - uv^2,\\
       v_t &= \Delta v - mv + uv^2.
       \end{aligned}
\end{equation} The isotropic spread of (surface) water $u$ is modelled by $d\Delta u$, downhill flow by $\beta u_x$, precipitation by $a$ and evaporation by $-u$. The uptake of water by vegetation $\pm uv^2$ is quadratic in the vegetation to model enhanced water infiltration at locations with vegetation. Vegetation dispersal is modelled by $\Delta v$ and mortality by $-mv$. We fix the parameters to customary values $d=500$ and $m=0.45$ and investigate how (small) advection impacts the patterns by `brute force' computing them and their stability against large-wavelength instabilities with pde2path \cite{Uecker2014}. For this we choose $\beta=0$ or $50$ or $100$, which are relatively small values \cite{Siero2015}. The parameter $a$ is chosen so that the system is near Turing(-Hopf) instability.  

A spatially homogeneous steady state is given by $(u,v)=(a,0)$ and for $a\geq 2m$ there are two more:
\begin{equation*}
       \begin{aligned}
	       u_\pm(a) =\frac{2m^2}{a\pm\sqrt{a^2-4m^2}}, \quad v_\pm(a) =\frac{a\pm\sqrt{a^2-4m^2}}{2m}.
       \end{aligned}
\end{equation*} From these two, only $(u_+,v_+)$ is stable against spatially homogeneous perturbations, and becomes Turing(-Hopf) unstable when $a$ drops below a critical value. 

In Fig.~\ref{f:Klausmeier} the large-wavelength stability of stripes near onset is depicted. Contrary to the previous example, for increasing advection $\beta$ the zigzag boundary shifts to the left of the Turing(-Hopf) instability, so in this case the stripes with critical wavenumber that emerge are zigzag-stable.

\begin{figure}[h!]
\centering
	\subfloat[$\beta=0$]{\includegraphics[clip=true, trim=70 0 65 5, width=0.33\linewidth]{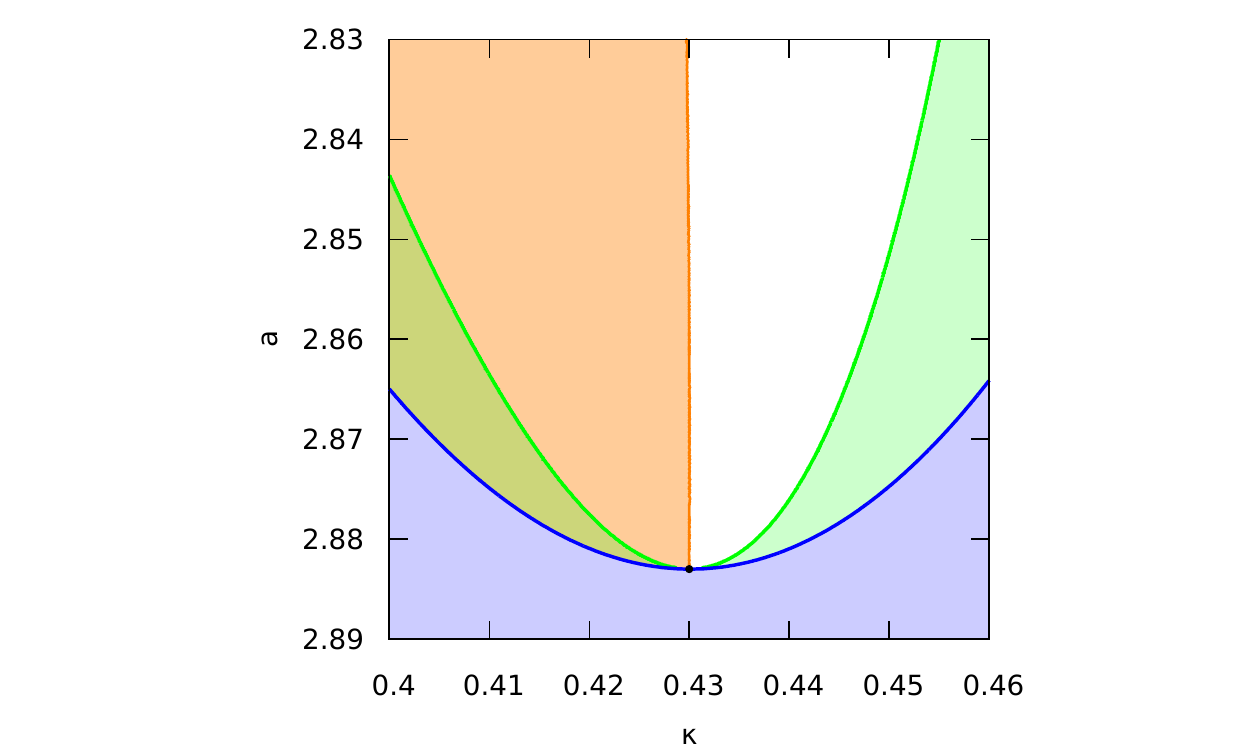}\label{f:Klausmeierbet0}}
	\subfloat[$\beta=50$]{\includegraphics[clip=true, trim=70 0 65 5, width=0.33\linewidth]{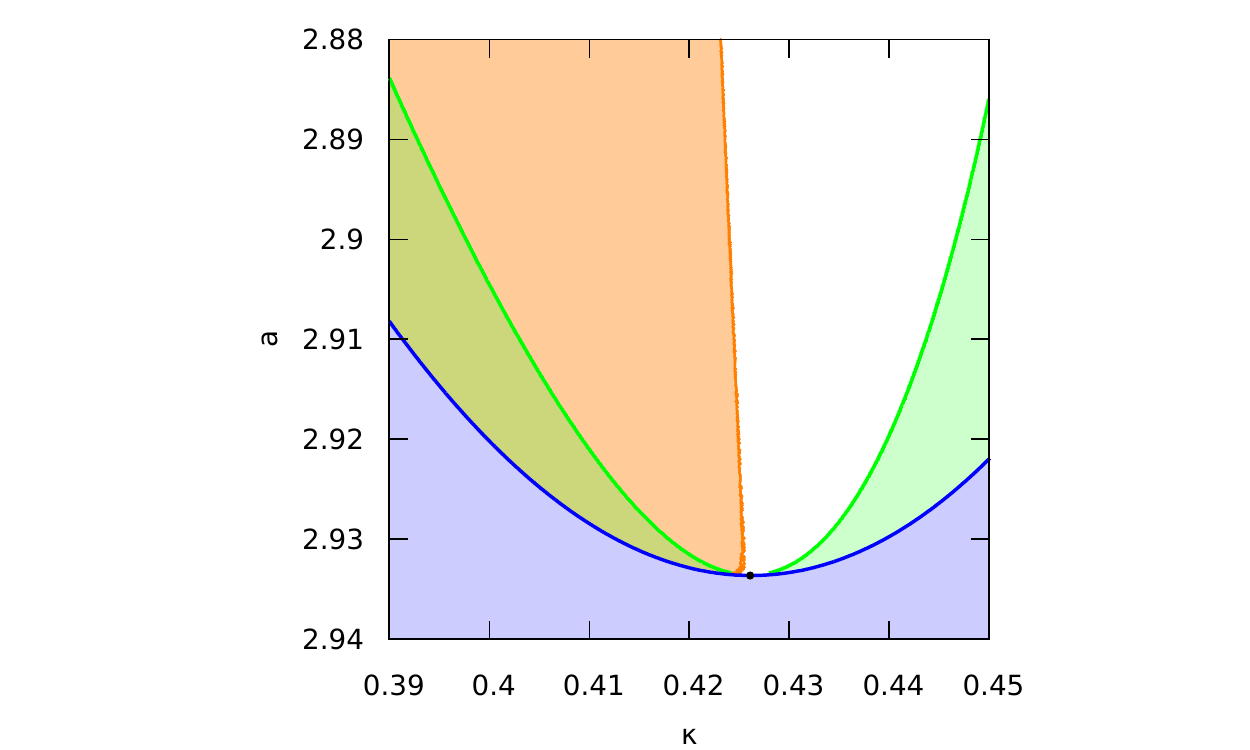}}
	\subfloat[$\beta=100$]{\includegraphics[clip=true, trim=70 0 65 5, width=0.33\linewidth]{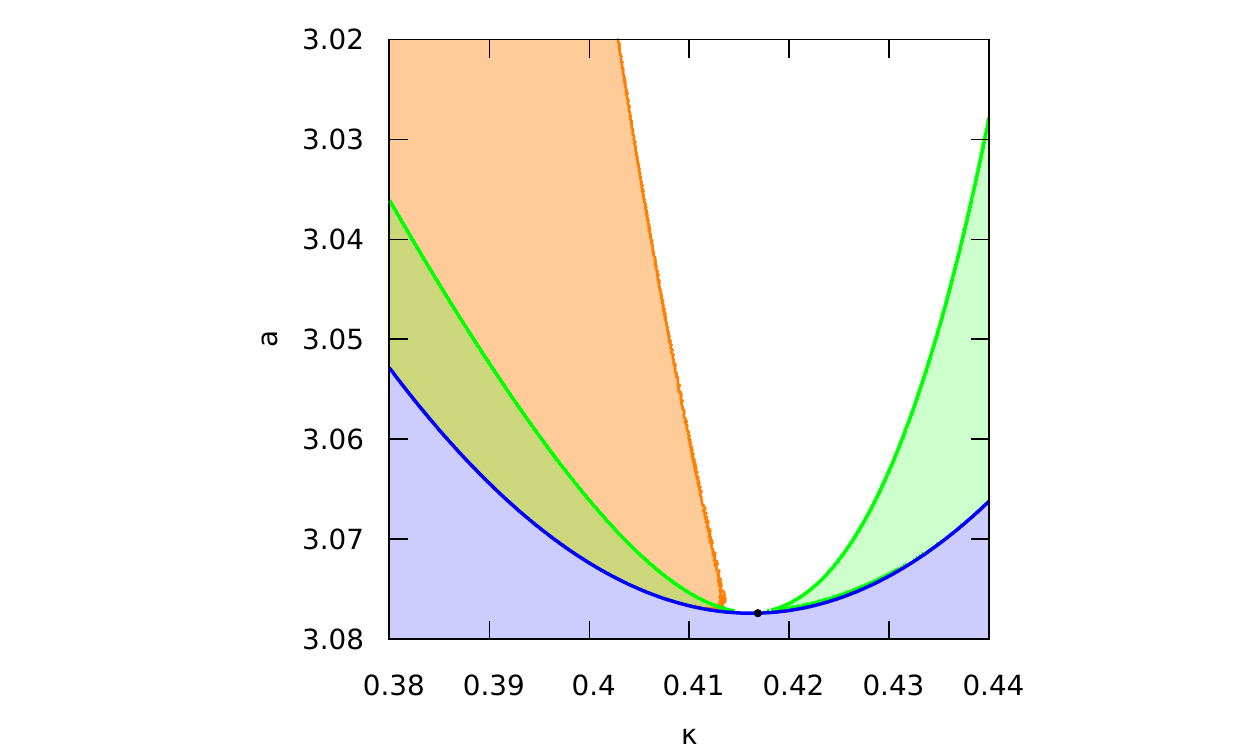}}
	\caption{Eckhaus and zigzag (in)stability regions of the stripes for \eqref{e:Klausmeier} in the $(\kappa,a)$-plane, by numerically checking the spectrum on an equidistant grid with spacing between neighboring grid points of $a=0.0001$ and $\kappa=0.0001$. Stripes exist in the complement of the blue region; green regions: Eckhaus unstable; orange regions: zigzag unstable. (a) $\beta=0$. (b) $\beta=50$. (c) $\beta=100$. In (a) the zigzag boundary is attached to the Turing instability locus and visually vertical, whereas in (b) and particularly (c) the zigzag boundary has shifted and tilted to the left.}\label{f:Klausmeier}
\end{figure}

In addition, we can transform \eqref{e:Klausmeier} into the framework of \eqref{e:RDS}. Since $(u_+,v_+)$ is a function of $a$, we consider $u=u_++\tilde u$, $v=v_++\tilde v$ such that the equilibrium is shifted to $(\tilde u,\tilde v) = (0,0)$. Removing the `tilde' yields
\begin{equation}
	\begin{aligned}
	u_t &= d\Delta u - (1+v_+^2)u - 2m v + \beta u_x - 2v_+ uv - u_+v^2 - uv^2,\\
	v_t &= \Delta v  + v_+^2 u + mv + 2v_+ uv + u_+v^2 +uv^2,
	\end{aligned}
\end{equation}
where the linear matrix, quadratic and cubic forms are given by
\[
\tilde\A = \begin{pmatrix}
-1-v_+^2 & -2m\\ v_+^2 & m
\end{pmatrix}, \; \Q[U,U]=\begin{pmatrix}
- 2v_+ uv - u_+v^2 \\  2v_+ uv + u_+v^2
\end{pmatrix},\; \K[U,U,U]=\begin{pmatrix}
-uv^2 \\ uv^2
\end{pmatrix},
\]
with $U:=(u,v)^T$.
In particular, the generic form of $\Q$ is given by $\Q[U_1,U_2] = (\Q_|[U_1,U_2],\Q_{||}[U_1,U_2])^T$ with
\[
-\Q_|[U_1,U_2] = \Q_{||}[U_1,U_2] = U_1^T \begin{pmatrix}
0 & v_+\\ v_+ & u_+
\end{pmatrix} U_2 ,
\]
where $U_j=(u_j,v_j)^T$, $j=1,2$. Since the Turing bifurcation occurs at $a=a_T$~\cite{Siero2015}, expanding $\tilde\A$ near $a=a_T$ yields
\[
\tilde\A = \tilde\A(a) = \tilde\A(a_T) + \partial_a\tilde\A(a_T)(a-a_T) + \calO((a-a_T)^2).
\]
In terms of the notations in \eqref{e:RDS}, we denote $\A:=\tilde\A(a_T)$, $\M:=\partial_a\tilde\A(a_T)$ and $\calpha:= a-a_T$. The higher order term $\calO(\calpha^2)$ is not relevant to the analysis in this paper, so we omit it. Therefore, we transform \eqref{e:Klausmeier} into the framework of \eqref{e:RDS}, and the analysis of \eqref{e:RDS} for `sufficiently small' $\calpha,\beta$ can be applied to \eqref{e:Klausmeier}. We list the leading order of the existence, zigzag and Eckhaus boundaries of the extended Klausmeier model \eqref{e:Klausmeier} for small $\mu$ as follows ($\alpha \approx -0.137\calpha$).
\begin{align*}
\text{bifurcation curve:}\quad & \alpha \approx -2.81\times 10^{-6}\beta^2+8.39\kap^2,\\
\text{Eckhaus boundary:}\quad & \alpha \approx -2.81\times 10^{-6}\beta^2+2.80\kap^2,\\
\text{zigzag boundary:}\quad & \alpha \approx -2.38\times10^{-5}\beta^2-3.09\kap.
\end{align*}
The advection shifts the bifurcation curve and the Eckhaus boundary downwards, and shifts the zigzag boundary to the left. Notably, the zigzag boundary has a negative slope and it is true for $\beta=0$ as well. This can be seen from Fig.~\ref{f:Klausmeierbet0} that the zigzag boundary is not precisely vertical.

\bigskip
\appendix

\section{Proof of Theorem \ref{t:bif}}\label{s:bifproof}
Writing the nonlinear part as $F(u) := \Q[u,u] + \K[u,u,u]$ we seek wave trains as steady states, i.e. solutions to
\begin{align}\label{e:RDSper}
\Phi(u,\mu):=\calL_\mu u + F(u) = 0
\end{align}
with $\Phi: (\Hspace^2_\per)^2 \times \Lambda \to (\Lspace^2)^2$ on the Sobolev- and Lebesgue-spaces $(\Hspace^2_\per)^2$ to $(\Lspace^2)^2$ with normalised inner product $\langle u,v\rangle_{\Lspace^2} = \frac 1{2\pi}\int_0^{2\pi}u\overline{v}\rmd x$. It is well-known that the realisation $\calL_\mu:(\Hspace^2_\per)^2 \to (\Lspace^2)^2$ is a bounded Fredholm operator with index zero. Therefore, all solutions to \eqref{e:RDSper} which bifurcate from $\mu=0$ can be fully determined by Lyapunov-Schmidt reduction. 

By assumption, $\calL_0$ has a two-dimensional kernel spanned by $e_0(x)=\E_0\rme^{\rmi x}$ and its complex conjugate, where $\E_0$ is the eigenvector in the kernel of $\hat\calL_0=-\kcsq D + L$. Let $\calL_0^*$ be the adjoint operator of $\calL_0$ equipped with inner product $\langle \cdot,\cdot\rangle_{\Lspace^2}$, and thus $\calL_0^*$ has a kernel spanned by $e_0^*(x) = \E_0^*\rme^{\rmi x}$ and its complex conjugate. Having in mind the scaled inner products, we choose the normalisation $\langle e_0, e_0^* \rangle = 1$ and $\langle e_0, e_0 \rangle = 1$, i.e., $\langle E_0, E_0^* \rangle = 1$ and $\langle E_0, E_0 \rangle =1$ (cf. Remark \ref{r:ev}).

By Fredholm properties there exists closed subspaces $X\subset(\Hspace^2_\per)^2$ and $Y\subset(\Lspace^2)^2$ such that
\begin{align*}
	(\Hspace^2_\per)^2 & = \ker \calL_0 \oplus X, \;
	(\Lspace^2)^2  = Y \oplus \range \calL_0.
\end{align*}
Hence for each $u \in (\Hspace^2_\per)^2$, there exists unique $v\in \ker\calL_0$ and $w\in X$ such that $u=v+w$.
With the projection $P_h: (\Lspace^2)^2 \to \range \calL_0$, equation \eqref{e:RDSper} is equivalent to the system
\begin{align}
	P_h \Phi(v + w, \mu) & = 0, \label{e:proj1}\\
	(\Id - P_h) \Phi(v + w, \mu) & = 0. \label{e:proj2}
\end{align}
Differentiating \eqref{e:proj1} with respect to $w$ at $(0,0)$ gives $P_h \partial_u \Phi(0,0) = P_h \calL_0 = \calL_0|_X: X \to \range \calL_0$ as a boundedly invertible operator. Hence, for given $v$ \eqref{e:proj1} can be solved by the implicit function theorem in terms of a smooth function $W:\ker \calL_0\times\Lambda\to Y$ with $W(0,0)=0$, $\partial_v W(0,0)=0$ as
\begin{align}\label{e:IFTsol}
	w = W(v,\mu),
\end{align}
satisfying $P_h\Phi(v+W(v,\mu),\mu)=0$. Substituting \eqref{e:IFTsol} into \eqref{e:proj2} yields the bifurcation equation
\begin{align*}
	\phi(v,\mu) := (\Id - P_h) \Phi(v + W(v,\mu), \mu) =0,
\end{align*}
with $\phi:\ker\calL_0\times\Lambda \rightarrow Y$.
Since $\range\calL_0\cap \ker \calL_0=\{0\}$ and we are in Hilbert spaces, we can choose 
\[
X=\range\calL_0 \cap (\Hspace^2_\per)^2, \quad Y = \ker \calL_0^* = (\range\calL_0)^\perp,
\]
where the adjoint $\calL_0^*$ has a kernel spanned by $e_0^*(x)= \E_0^* \rme^{\rmi x}$ and its complex conjugate. Hence, it is natural to write the projection as $P_h=\Id- P$ with
\[
Pu = \langle u, e_0^*\rangle e_0 + \langle u, \overline{e_0^*}\rangle \overline{e_0},
\]
which equally is a projection for the splitting $u=v+w$, when constrained to $(\Hspace^2_\per)^2$. With some abuse, we use the same notation for inner products in $\Lspace^2$ and $\C^2$ as it is clear from the context what is meant. Note that $\langle P\Phi, e_0^* \rangle = \langle \Phi, e_0^* \rangle$ since $\langle P_h\Phi, e_0^* \rangle=0$ for solutions. 

Writing $v=Ae_0 + \overline{Ae_0}$ the bifurcation equation can be cast as
\begin{align}\label{e:bifeq}
g(A,\overline{A},\mu):= \langle \phi(Ae_0 + \overline{Ae_0},\mu), e_0^*\rangle = 0
\end{align}
with $g:\C\times\C\times\Lambda\rightarrow\R$ which we next expand in order to expand solutions.
Using \eqref{e:RDSper} and $P\calL_0=0$ gives 
\begin{align}\label{e:gsplit}
g(A,\overline{A},\mu) = \langle(\calL_\mu-\calL_0)v,e_0^*\rangle + \langle(\calL_\mu-\calL_0)W,e_0^*\rangle 
+ \langle F(v+W),e_0^*\rangle
\end{align}
Let us first consider the last term that includes $F$. While $\E_0$, $\E_0^*$ are real, we show the complex conjugate to highlight the origin of terms. It is a priori clear from the construction  that  $W=\calO(|v||\mu|) =\calO(|A||\mu|)$, cf.\ \eqref{e:Wexp} for the details a posteriori. For $u=Ae_0 + \overline{Ae_0}+W$ we then readily compute
\begin{align}\label{e:Kexp}
\langle K[u,u,u],e_0^*\rangle = 3A|A|^2 \hK + \calO(|A|^3(|\mu|+A^2)),
\end{align}
where $\hK=\langle \K[\E_0,\E_0,\overline{\E_0}], \E_0^*\rangle$, and we used that orthogonality of Fourier modes removes even powers of $v$, i.e., even powers of $A$.

The more involved $\langle Q[u,u],e_0^*\rangle$ analogously gives 
\begin{align}\label{e:WQ}
2A\langle Q[e_0,W],e_0^*\rangle + 2A\langle Q[\overline{e_0},W], e_0^* \rangle + \calO(A^2(|\mu|^2+A^2)),
\end{align}
which requires expanding $W=W(A,\overline{A},\mu)$ through \eqref{e:proj1}, i.e., the fixed point equation $P_h\calL_\mu W=G(W,A,\overline{A},\mu)$ with
\begin{equation}\label{e:GW}
G(W,A,\overline{A},\mu) := -P_h F(v+W) - P_h (\calL_\mu - \calL_0) v,
\end{equation}
where  $v=Ae_0 + \overline{Ae_0}$ and
\[
\calL_\mu - \calL_0 = (2 \kc \kap + \kap^2)D\partial_x^2 + \calpha M + \beta(\kc+\kap) \B\partial_x.
\]
Using $\partial_v G(0)=\partial_W G(0)$ this expansion gives
$\partial_AW(0) = \partial_{\overline{A}}W(0) = 0$ and, cf.\ \eqref{e:defs},
\begin{alignat*}{3}
	&\partial_{AA}W(0) &&= -2(-4 \kcsq D + \A)^{-1} \Q[\E_0,\E_0]=\hQ, \\
	&\partial_{A\overline{A}}W(0) &&= -2\A^{-1} \Q[\overline{\E_0},\E_0]=\tQ,\\
	&\calL_0\partial_{A\calpha} W(0) &&= -P_h  M e_0,\\
	&\calL_0\partial_{A\beta} W(0) &&=-\kc P_h \B \partial_x e_0 \;=-\rmi\kc P_h \B e_0,\\ 
	&\calL_0\partial_{A\kap} W(0) &&= -2\kc P_h D\partial_x^2e_0 \;= 2\kc P_h D e_0,\\
	&\calL_0\partial_{A\beta\beta} W(0) &&=-2\kc P_h \B \partial_x \partial_{A\beta}W(0),
\end{alignat*}
so that, looking at the Fourier modes,
\begin{align*}
	\partial_{A\beta} W(0) = \rmi w_{A\beta}\rme^{\rmi x},\,
	\partial_{A\kap} W(0) = w_{A\kap}\rme^{\rmi x}, \,
	\partial_{A\calpha} W(0) = w_{A\calpha}\rme^{\rmi x},\,
	\partial_{A\beta\beta} W(0) = w_{A\beta\beta}\rme^{\rmi x}.
\end{align*}
Furthermore $w_{A\beta}, w_{A\kap}$ and $w_{A\calpha}$ satisfy, cf.\ \eqref{e:defs}, 
\begin{alignat*}{3}
	&(-\kcsq D+\A)w_{A\calpha}&& = (\langle M E_0, E_0^*\rangle- M) \E_0,\\
	&(-\kcsq D+\A) w_{A\beta} &&= \kc(\langle \B\E_0,\E_0^*\rangle -\B)\E_0,\\
	&(-\kcsq D+\A)w_{A\kap}&& =  2\kc D \E_0,\\
	&(-\kcsq D+\A) w_{A\beta\beta}&&= 2\kc (\B w_{A\beta} - \langle \B w_{A\beta},E_0^*\rangle\E_0),
\end{alignat*}
where we used $\langle D \E_0, \E_0^* \rangle = 0$, which follows from a direct computation with the conditions in Remark~\ref{r:Turing}. Note that for $M={\rm Id}$ we have $w_{A\calpha}=0$ and in fact $W$ is independent of $\calpha$.

Assembling terms, we obtain
\begin{equation}\label{e:Wexp}
\begin{aligned}
W(A,\overline{A},\mu) =\ & \rmi\beta w_{A\beta} (A \rme^{\rmi x} - \overline{A}\rme^{-\rmi x})\\
&+ (\kap w_{A\kap} + \calpha w_{A\calpha} + \beta^2 w_{A\beta\beta})(A \rme^{\rmi x} + \overline{A}\rme^{-\rmi x})\\
&+ \frac{1}{2} \hQ\left( A^2 \rme^{2 \rmi x} + \overline{A}^2 \rme^{-2 \rmi x}\right) + A\overline{A}\tQ+ \calR,
\end{aligned}
\end{equation}
where $\calR=\calO(|A|(A^2 + \kap^2 + |\beta\kap| + |\beta|^3  + a_M\calpha^2))$; recall $a_M=0$ if $M={\rm Id}$ and $a_M=1$ otherwise. Notably, the terms of order $|A\beta\kap|$ are not relevant to the large-wavelength stability of stripes, so we put them in the remainder.
By translation symmetry we can shift $x$ to $x+a$, which gives $A$ replaced by $A\rme^{\rmi a}$ so that without loss of generality $A$ is real. This gives \eqref{e:Stripes}.

\medskip
The bifurcation equation~\eqref{e:bifeq} gives \eqref{e:stripeeqn} through its real part divided by $A$. The velocity equation \eqref{e:stripev} stems from rearranging the imaginary part divided by $\beta A$. The latter is natural since imaginary terms arise from odd powers of $\partial_x$, which come with odd power of $\beta$.
In order to separate resolved parts, that will be leading order for later purposes, and remainder terms in \eqref{e:stripeeqn}, \eqref{e:stripev}, we substitute \eqref{e:Wexp} into \eqref{e:gsplit}, where the third summand is \eqref{e:Kexp} plus \eqref{e:WQ}. 
In \eqref{e:WQ} only Fourier modes $e^{i r x}$ of $W$ with $r=0$ or $r=2$ are nonzero. The case $r=0$ stems only from products that have $A^j\overline{A}^j$, $j\geq 2$ as the order in $A$ since $j=0,1$ are resolved terms; the case $r=2$ has $A^{j+2}\overline{A}^j$, $j\geq 1$. Hence, terms in \eqref{e:WQ} that stem from $\calR$ are order $A^4$; from \eqref{e:Kexp} this is order $A^5$. The second summand of \eqref{e:gsplit} is nonzero for linear terms in $A$ only, which contribute to higher order terms in $\tlam(\mu)$ as discussed below. Hence, the relevant remainder term from \eqref{e:gsplit} is order $A^4$.

The remainder term in \eqref{e:stripeeqn} and \eqref{e:stripev} has this order divided by $A$, i.e., $\calO(A^3)$. This is also the order of the contribution of $\calR$ to the remainder term in \eqref{e:stripev}; here we note that real parts turn imaginary in \eqref{e:gsplit} only through application of $\beta \kc \B\partial_x$ thus gaining a power of $\beta$. 

\medskip
The part of \eqref{e:bifeq} that is resolved in \eqref{e:stripeeqn}, \eqref{e:stripev} arises upon substituting the resolved terms of \eqref{e:Wexp} into \eqref{e:WQ}, and further into \eqref{e:gsplit}; here \eqref{e:Kexp} directly enters. Noting cancellation due to the Fourier modes and dividing out the trivial solution $A=0$ we obtain
\begin{equation}\label{e:bifeqn}
	\partial_A g(0;\mu) + \rho_\nl A^2,
\end{equation}
and its complex conjugate.  On the one hand, for $F=0$ the bifurcation equation is linear in $A$ and determines when $\calL_\mu$ has a kernel, which means there is a smooth function $r(\mu)$ such that
\[
\partial_A g(0;\mu)= \tlam(\mu) = r(\mu)\lambda_\mu,
\]
with $\lambda_\mu$ the critical eigenvalue from \eqref{e:evlinearop} and $r(0)\neq 0$. Expanding $r$ we thus have
\[
\tlam(\mu) = r(0)\lambda_\mu + \calR_3, \; \calO(\calR_3) = \calO(|\mu|\, |\lambda_\mu|)
\]
and we can determine the expansion of $\tlam$ from 
\begin{align*}
\partial_A g(0;\mu) &= \langle (\calL_\mu-\calL_0)(\Id+\partial_{\mu}\partial_v W\mu+ \calO(|\mu|^2))e_0,e_0^*\rangle\\
&= \langle(\calL_\mu-\calL_0)e_0,e_0^*\rangle + \langle(\calL_\mu-\calL_0)\partial_{\mu}\partial_v W\mu+ \calO(|\mu|^2))e_0,e_0^*\rangle.
\end{align*}
In particular, $\partial_{\calpha} \tlam(0) = \langle M E_0, E_0^*\rangle$ which equals $\lambda_\M$ by a direct computation so that $r(0)=1$. Moreover, the real part of \eqref{e:bifeqn} gives \eqref{e:stripeeqn} and solving the imaginary part divided by $\beta$ for $c$ gives \eqref{e:stripev} when including the remainder terms discussed before. In particular, $r(0)=1$ yields \eqref{e:stripecoeffalt} by comparing the other coefficients, and $\langle BE_0, E_0^*\rangle|_{\mu=0}=0$.

\bigskip
\appendixnotitle

\subsection{Proof of Theorem \ref{t:zigzag}}\label{s:zigzagproof}
From \eqref{e:seconddifell} the critical spectrum is given by
\[
\lambda_{\rm zz} = -\kappa^2\langle DV_0, V_0^* \rangle \ell^2 + \calO(\ell^4).
\]
We may choose $V_0 =\partial_xU_\rms = \calO(|A|)$. Expanding $\calT_0^* V_0^*=0$ analogous to the computation of $U_\rms$ gives
\begin{align*}
	V_0 =\ &-2A \big((\E_0 + \kap w_{A\kap}+\calpha w_{A\calpha}+\beta^2w_{A\beta\beta})\sin(x) + \beta w_{A\beta}\cos(x)\\
	&\qquad\ \ + A\hQ\sin(2x) + \calO(\calR/|A|)\big),\\
	V_0^* =\ &-A^*\big((\E_0^* + \kap w_{A\kap}^*+\calpha w_{A\calpha}^*+\beta^2w_{A\beta\beta}^*)\sin(x) - \beta w_{A\beta}^*\cos(x)\\
	&\qquad\ \ + A\hQ^*\sin(2x) + \calO(\calR/|A|)\big),
\end{align*}
where 
\begin{align*}
	w_{A\kap}^* &= 2\kc(-\kcsq D+\A^T)^{-1} D\E_0^*,\\
	w_{A\beta}^* &= \kc(-\kcsq D+\A^T)^{-1} (\langle \B\E_0^*,\E_0 \rangle - \B)\E_0^*,\\
	w_{A\calpha}^* &= (-\kcsq D+\A^T)^{-1} (\langle \M^T E_0^*, E_0\rangle- \M^T) \E_0^*,\\
	w_{A\beta\beta}^* &= 2\kc(-\kcsq D+\A^T)^{-1} (\B w_{A\beta}^* - \langle \B w_{A\beta}^*,E_0\rangle\E_0^*),\\
	\hQ^* &= -2(-4\kcsq D+\A^T)^{-1}\Q[\E_0,\cdot]^T\E_0^*.
\end{align*}
The normalised coefficient $A^*$ is such that $\langle V_0,V_0^* \rangle = 1$, which implies $A^*=\calO(|A|^{-1})$ and $AA^*=1$ in the limit $\mu\to 0$ since $\langle V_0,V_0^* \rangle|_{\mu=0} =AA^*|_{\mu=0} \langle \E_0,\E_0^* \rangle =1$. By straightforward calculation and using $\langle D\E_0,\E_0^* \rangle =0$, we have
\begin{align}
	\frac{1}{AA^*}\langle DV_0, V_0^* \rangle =\ &\kap(\langle D\E_0, w_{A\kap}^* \rangle + \langle Dw_{A\kap}, \E_0^* \rangle)\nonumber\\
	&+ \beta^2(\langle D\E_0,w_{A\beta\beta}^*\rangle +\langle Dw_{A\beta\beta},E_0^*\rangle-\langle Dw_{A\beta}, w_{A\beta}^* \rangle)\nonumber\\
	&+ \calpha(\langle D\E_0,w_{A\calpha}^*\rangle + \langle D w_{A\calpha},\E_0^*\rangle)\nonumber\\
	&+ A^2\langle D\Q_2,\Q_2^* \rangle+ \calO(\calR/|A|)\nonumber\\
	=\ &-\frac{\rho_\kap}{\kc}\kap  +\tilde\rho_{\bbeta} \beta^2 + \tilde\rho_\calpha a_\M\alpha + \tilde q_{22}A^2 + \calO(\calR/|A|),\label{e:seconddifzigzag}
\end{align}
where $\langle Dw_{A\kap}, \E_0^* \rangle = \langle D\E_0, w_{A\kap}^* \rangle = -\rho_\kap/(2\kc)$. Upon substitution into $\lambda_\zz$, expansion of $\kappa$ and using the leading order of \eqref{e:stripeeqn} yields the claimed result.

\subsection{Proof of Theorem \ref{t:Eckhaus}}\label{s:Eckhausproof}
The critical spectrum is given by 
\[
\lambda_\eh = (\partial_\gamma\lambda)_0 \gamma + \frac{1}{2}(\partial_\gamma^2\lambda)_0\gamma^2 + \calO(|\gamma|^3).
\]
We first compute $(\partial_\gamma\lambda)_0$, i.e., the terms in \eqref{e:gamfirstdiff} and \eqref{e:seconddigfam}. Differentiating
\[
\calL_\mu U_\rms + \Q[U_\rms,U_\rms] + \K[U_\rms,U_\rms,U_\rms] = 0
\]
with respect to $\kap$ and rearranging terms yields
\begin{align}
	\calT_0 \partial_{\kap} U_\rms = - 2(2\kappa D\partial_x + \beta\B + \beta\kappa\partial_{\kap}\B)V_0.\label{e:dkapUs}
\end{align}
Hence, we can solve for $\partial_{\kap}U_\rms$ if and only if 
\[
\langle (2\kappa D\partial_x + \beta\B + \beta\kappa\partial_{\kap}\B)V_0, V_0^* \rangle = 0,
\]
where $\partial_{\kap}\B = \partial_{\kap}c\cdot\Id = (\lambda_{\beta}-\lambda_{\kap\beta})/\kc\cdot\Id$, cf. \eqref{e:stripev}, so that from \eqref{e:gamfirstdiff} we have 
\begin{align}\label{e:lamgam}
	(\partial_\gamma\lambda)_0 = -\rmi\kappa\langle \beta\kappa\partial_{\kap}\B V_0,V_0^* \rangle = -\rmi\kappa^2\beta\partial_{\kap}c = \rmi\kappa^2\beta\frac{\lambda_{\kap\beta}-\lambda_{\beta}}{\kc}, 
\end{align}
and the leading order gives the imaginary part of the claimed spectrum.

\begin{remark}\label{r:groupv}
	As in \cite{Doelman2009} $(\partial_\gamma\lambda)_0$ measures the correction of the phase velocity $c$ to the group velocity $c_{\rm g}$. Let $\omega(\kappa)$ denote the nonlinear dispersion relation so that
	\[
	c = \frac{\omega(\kappa)}{\kappa},\quad c_{\rm g} = \frac{\dif \omega(\kappa)}{\dif\kappa}.
	\]
	Differentiating $c$ with respect to $\kap$ gives
	\[
	\partial_{\kap}c = \frac{1}{\kappa}\frac{\dif \omega(\kappa)}{\dif\kappa} - \frac{\omega(\kappa)}{\kappa^2} = \frac{c_{\rm g} - c}{\kappa},
	\]
	and substituting into \eqref{e:lamgam}, yields
	\[
	(\partial_\gamma\lambda)_0 = \rmi\kappa\beta(c - c_{\rm g}).
	\]
	Hence, for $\beta\neq0$, $0\leq |\kap|\ll 1$, we have $(\partial_\gamma\lambda)_0=0 \Leftrightarrow c=c_{\rm g}$.
\end{remark}

\medskip
Next, we consider $(\partial_\gamma V)_0$. Due to \eqref{e:firstdiff}, \eqref{e:Tgam} and \eqref{e:dkapUs} we have
\begin{align*}
	\calT_0(\partial_\gamma V)_0 &= (\partial_\gamma\lambda)_0V_0 - (\partial_\gamma \calT)_0 V_0 = (-\rmi\kappa^2\beta\partial_{\kap}c - 2\rmi\kappa^2 D\partial_x - \rmi\kappa\beta\B)V_0\\
	&= \rmi\kappa\calT_0 \partial_{\kap} U_\rms = \calT_0 (\rmi\kappa\partial_{\kap} U_\rms),
\end{align*}
which implies that $(\partial_\gamma V)_0 - \rmi\kappa\partial_{\kap}U_\rms$ lies in the kernel of $\calT_0$, spanned by $V_0$, and there is $a\in\C$ such that $(\partial_\gamma V)_0 = \rmi\kappa\partial_{\kap}U_\rms + aV_0$.

This term is not relevant for $(\partial_\gamma^2\lambda)_0$ since we compute
\begin{align*}
	(\partial_\gamma^2\lambda)_0 =\ & \langle (\partial^2_\gamma\calT)_0V_0,V_0^*\rangle - 2\langle ((\partial_\gamma\lambda)_0 - (\partial_\gamma\calT)_0)(\partial_\gamma V)_0, V_0^* \rangle\\
	=\ & -2\kappa^2\langle DV_0,V_0^* \rangle - 2\langle ((\partial_\gamma\lambda)_0 - (\partial_\gamma\calT)_0)(\rmi\kappa\partial_{\kap}U_\rms + aV_0), V_0^* \rangle\\
	=\ & -2\kappa^2\langle DV_0,V_0^* \rangle - 2\langle ((\partial_\gamma\lambda)_0 - (\partial_\gamma\calT)_0)(\rmi\kappa\partial_{\kap}U_\rms), V_0^* \rangle
\end{align*}
using  \eqref{e:firstdif} in the third equality. Upon substituting \eqref{e:lamgam} and 
\eqref{e:Tgam} we obtain 
\begin{align}
	(\partial_\gamma^2\lambda)_0 =
	\ &- 2\kappa^2\langle 2\kappa D\partial_x \partial_{\kap}U_\rms, V_0^*\rangle \label{e:seconddiff}\\
	& -2\kappa^2\langle DV_0,V_0^* \rangle + 2\beta\frac{\kappa^3}{\kc}(\lambda_{\kap\beta}-\lambda_{\beta})\langle \partial_{\kap}U_\rms, V_0^* \rangle - 2\kappa^2\beta\langle \B\partial_{\kap}U_\rms, V_0^* \rangle,\label{e:seconddiffhot}	
\end{align}
and we will show that \eqref{e:seconddiff} is leading order. 
By \eqref{e:seconddifzigzag} the first term in \eqref{e:seconddiffhot} is order $\calO(|a_\M\alpha| + |\kap| + \beta^2+A^2)$, and we show the others are $\calO(A^{-2}\beta^2(|\kap|+|a_\M\alpha|)+\beta^2)$.

Differentiating $U_\rms$ with respect to $\kap$ gives
\begin{align*}
	\partial_{\kap}U_\rms =\ & \partial_{\kap}A\left( 2(\E_0+\kap w_{A\kap} + \calpha w_{A\calpha} + \beta^2 w_{A\beta\beta})\cos(x) - 2\beta w_{A\beta}\sin(x)\right)\\ 
	&+ 2A w_{A\kap}\cos(x)+ 2A\partial_{\kap}A\Q_2\cos(2x) + 2A\partial_\kap A\Q_0 + \partial_{\kap}\calR,
\end{align*}
with $\partial_{\kap}\calR=\calO(|A|(|\beta|+|\kap|)+|\partial_{\kap}A|\calR/|A|)$ by differentiating the smooth remainder in \eqref{e:Stripes}. In the following we frequently omit remainder terms such as $\calR$ as the order of the remainder terms do not change and we are only interested in the resolved terms, which will be higher order in the application of the result. From Theorem \ref{t:bif}, 
\begin{align}\label{e:Akap}
A_{\kap} := \partial_{\kap}A = -\frac{2\rho_\kap\kap+\lambda_{\M\kap}a_\M\alpha}{2\rho_\nl A}= \calO(|A|^{-1}(|\kap|+|a_\M\alpha|)),
\end{align}
which means 
\begin{align*}
	\langle \partial_{\kap}U_\rms, V_0^* \rangle =\ &A_{\kap}A^*\beta(\langle w_{A\beta}, \E_0^* \rangle + \langle \E_0, w_{A\beta}^* \rangle) + AA^*\beta \langle w_{A\kap}, w_{A\beta}^* \rangle \\
	\langle \B\partial_{\kap}U_\rms, V_0^* \rangle =\ &A_{\kap}A^*\beta(\langle \B w_{A\beta}, \E_0^* \rangle + \langle \B\E_0, w_{A\beta}^* \rangle) + AA^*\beta \langle \B w_{A\kap}, w_{A\beta}^* \rangle,
\end{align*}
so that \eqref{e:seconddiffhot} is of order $\calO((1+A^{-2}\beta^2)(|\kap|+|a_\M\alpha|)+\beta^2 +A^2)$.

As to \eqref{e:seconddiff}, differentiating $\partial_{\kap}U_\rms$ with respect to $x$ gives (to leading order)
\begin{align*}
	\partial_x\partial_{\kap}U_\rms =\ & -2A_\kap\left( (\E_0+\kap w_{A\kap} + \calpha w_{A\calpha} + \beta^2 w_{A\beta\beta})\sin(x) + \beta w_{A\beta}\cos(x)\right)\\ 
	&- 2A w_{A\kap}\sin(x) - 4AA_\kap\Q_2\sin(2x),
\end{align*}
thus we have
\begin{align}
	\langle D\partial_x\partial_{\kap}U_\rms, V_0^* \rangle =\ &(A_{\kap}A^*\kap + AA^*)\langle Dw_{A\kap}, \E_0^* \rangle + A_{\kap}A^*\kap\langle D\E_0, w_{A\kap}^* \rangle\nonumber\\
	&+ A_\kap A^*\calpha(\langle D w_{A\calpha},\E_0^* \rangle + \langle D\E_0,w_{A\calpha}^*\rangle)\nonumber\\
	&+ A_\kap A^*\beta^2(\langle D w_{A\beta\beta},\E_0^* \rangle+\langle D \E_0,w_{A\beta\beta}^* \rangle-\langle D w_{A\beta},w_{A\beta}^* \rangle)\nonumber\\
	&+ 2A^2A_\kap A^*\langle D\Q_2,\Q_2^*\rangle\nonumber\\
	=\ &(2A_{\kap}A^{-1}\kap + 1)\langle Dw_{A\kap}, \E_0^* \rangle + A_\kap A^{-1}(\rho_\calpha a_\M\alpha + \rho_{\bbeta}\beta^2)\label{e:seconddiflead}\\
	& + 2A^2A_\kap A^*\langle D\Q_2,\Q_2^*\rangle\nonumber.
\end{align}
Since $A^2A_\kap A^* = \calO(|\kap|+|a_\M\alpha|)$, it is a higher order term compared to $\langle Dw_{A\kap}, \E_0^* \rangle=\calO(1)$. 
Substituting \eqref{e:Akap}, \eqref{e:defs} and \eqref{e:stripeeqn}, \eqref{e:seconddiflead} becomes
\begin{align*}
	\langle D\partial_x\partial_{\kap}U_\rms, V_0^* \rangle =\ &\frac{\rho_\kap}{2\kc\rho_\nl}A^{-2}\left(\alpha + \rho_\beta\beta^2 + 3\rho_\kap\kap^2\right)\\
	 &+ \calO\left(A^{-2}(|a_\M\alpha|+|\kap|)(a_\M\alpha+\beta^2)+|a_\M\alpha|+|\kap|\right)
\end{align*}
Altogether, using $\kappa=\kc+\kap$ we have, omitting the refinement when $\M=\Id$,
\begin{align}
	(\partial_\gamma^2\lambda)_0 =&\ -4\kappa^3\langle D\partial_x(\partial_{\kap}U_\rms), V_0^* \rangle\\
	& + \calO(A^{-2}(|\alpha|+|\kap|)(|\alpha|+\beta^2)+|\alpha|+|\kap|+\beta^2+A^2)\nonumber\\
	=&\ -2\kcsq\frac{\rho_\kap}{\rho_\nl}A^{-2}\left(\alpha+\rho_\beta\beta^2+3\rho_\kap\kap^2 + \calR_\eh\right)\label{e:seconddiflam}
\end{align}
which is as claimed and has remainder term
\begin{align}\label{e:ehhot}
\calR_\eh=\calO\left(|\alpha\kap| + \alpha^2+|\alpha|\beta^2 + |\kap|\beta^2+|\kap|^3 + A^2(|\kap|+|\alpha|+\beta^2+A^2)\right).
\end{align}
Note that $A^2|\kap|$ is higher order compared to $\alpha+\rho_\beta\beta^2+3\rho_\kap\kap^2$ due to \eqref{e:ampfull},
and $|\alpha\kap|$ is higher order since $\alpha$ behaves quadratically for any balanced order between $\alpha,\beta^2, \kap^2$ which makes $|\alpha\kap|$ cubic order.

\bigskip
\section*{Acknowledgements}
J.Y. is grateful for the hospitality and support from Faculty 3 -- Mathematics, University of Bremen as well as travel support through an Impulse Grants for Research Projects by University of Bremen.


\end{document}